\documentclass[11pt]{amsart}
\usepackage{amsmath, amssymb}
\usepackage{amsfonts}
\usepackage[arrow,matrix,curve,cmtip,ps]{xy}

\usepackage{amsthm}

\allowdisplaybreaks

\newtheorem{theorem}{Theorem}[section]
\newtheorem{lemma}[theorem]{Lemma}
\newtheorem{proposition}[theorem]{Proposition}
\newtheorem{corollary}[theorem]{Corollary}

\newtheorem*{theorem*}{Theorem}
\theoremstyle{remark}
\newtheorem{remark}[theorem]{Remark}
\newtheorem{definition}[theorem]{Definition}
\newtheorem{example}[theorem]{Example}

\numberwithin{equation}{section}

\newcommand{\Z}{\mathbb{Z}}
\newcommand{\N}{\mathbb{N}}
\newcommand{\C}{\mathbb{C}}
\newcommand{\T}{\mathbb{T}}

\newcommand{\K}{\mathcal{K}}

\newcommand{\LC}{L_\mathbb{C}}

\newcommand{\reg}{\textnormal{reg}}

\newcommand{\ME}{\sim_{\textnormal{ME}}}
\newcommand{\SME}{\sim_\textnormal{{SME}}}

\newcommand{\coker}{\operatorname{coker}}
\newcommand{\algspan}{\operatorname{span}}
\newcommand{\clspan}{\operatorname{\overline{\textnormal{span}}}}
\newcommand{\im}{\operatorname{im }}
\newcommand{\Hom}{\operatorname{HOM}}
\newcommand{\Homalg}{\operatorname{HOM}_\textnormal{alg}}
\newcommand{\Iso}{\operatorname{ISO}}
\newcommand{\Isoalg}{\operatorname{ISO}_\textnormal{alg}}

\begin{document}
\title{Isomorphism and Morita Equivalence of Graph Algebras}

\author{Gene Abrams}

\author{Mark Tomforde}

\address{Department of Mathematics \\ University of Colorado \\ Colorado Springs, CO 80933 \\USA}
\email{abrams@math.uccs.edu}

\address{Department of Mathematics \\ University of Houston \\ Houston, TX 77204-3008 \\USA}
\email{tomforde@math.uh.edu}

\date{\today}

\subjclass[2000]{16D70, 46L55}

\keywords{graph, Leavitt path algebra, graph $C^*$-algebra, Morita equivalence}

\begin{abstract}

For any countable graph $E$, we investigate the relationship between the Leavitt path algebra $L_{\C}(E)$ and the graph $C^*$-algebra $C^*(E)$.  For graphs $E$ and $F$, we examine ring homomorphisms, ring $*$-homomorphisms, algebra homomorphisms, and algebra $*$-ho\-mo\-morph\-isms between $L_{\C}(E)$ and $L_{\C}(F)$. We prove that in certain situations isomorphisms between $L_{\C}(E)$ and $L_{\C}(F)$ yield $*$-isomorphisms between the corresponding $C^*$-algebras $C^*(E)$ and $C^*(F)$.  Conversely, we show that $*$-iso\-morph\-isms between $C^*(E)$ and $C^*(F)$ produce isomorphisms between $L_{\C}(E)$ and $L_{\C}(F)$ in specific cases.  The relationship between Leavitt path algebras and graph $C^*$-algebras is also explored in the context of Morita equivalence.

\end{abstract}

\maketitle

\tableofcontents

\section{Introduction}

For any directed graph $E$ one can define the graph $C^*$-algebra $C^*(E)$, which is generated by partial isometries satisfying relations determined by $E$.  These graph $C^*$-algebras include many well-known classes of $C^*$-algebras (e.g., Cuntz-Krieger algebras, stable AF-algebras, stable Kirchberg algebras, finite-dimensional $C^*$-algebras, $M_n(C(\T))$), and consequently they have been the focus of significant investigation in functional analysis since their introduction in the late 1990's \cite{KPR, KPRR}.    Similarly, for any directed graph $E$ and any field $K$ one can define the {\it Leavitt path algebra} $L_K(E)$.  Such $K$-algebras include many well-known classes of algebras and have been studied intensely in the algebra community since their introduction in 2005 \cite{AbrPino, AMP}.   The interplay between these two classes of ``graph algebras" has been extensive and mutually beneficial --- graph $C ^*$-algebra results have helped to guide the development of Leavitt path algebras by suggesting what results are true and in what direction investigations should be focused, and Leavitt path algebras have given a better understanding of graph $C^*$-algebras by helping to identify those aspects of $C^*(E)$ that are algebraic, rather than $C^*$-algebraic, in nature.

It has also been found that there are amazing similarities between the two classes of graph algebras.  In fact, every theorem from each class seems to have a corresponding theorem in the other.  At the same time, however, the similarities between various structural properties of $C^*(E)$ and $L_K(E)$ are as mysterious as they are startling.  For example, for nearly every graph-theoretic condition of $E$ that is known to be equivalent to a $C^*$-algebraic property of $C^*(E)$, the same graph-theoretic property of $E$ is equivalent to the corresponding property of $L_K(E)$.  For instance, the graph-theoretic conditions for which $C^*(E)$ is a simple algebra (respectively, an AF-algebra, a purely infinite simple algebra, an exchange ring, a finite-dimensional algebra) in the category of $C^*$-algebras are precisely the same graph-theoretic conditions for which $L_K(E)$ is a simple algebra (respectively, an ultramatricial algebra, a purely infinite simple algebra, an exchange ring, a finite-dimensional algebra) in the category of $K$-algebras.  Moreover, these Leavitt path algebras results  hold independent of the field $K$, and in particular for the field $K=\C$ of complex numbers.   These similarities might suggest that such structural properties, once obtained on either the graph $C^*$-algebra side or on the Leavitt path algebra side, might then immediately be translated via some sort of Rosetta stone  to the corresponding property on the other side.   Nonetheless, a vehicle to transfer information in this way remains elusive, and in fact, researchers seem uncertain how to even formulate conjectures that would lead to such a vehicle.

The purpose of this article is to initiate a study for translating properties of Leavitt path algebras to graph $C^*$-algebras.  We accomplish this by further examining the relationship between these classes and posing two conjectures.  We hope that these results will be useful in their own right, as well as help to lay the groundwork for future investigations.

Much of our focus will be on the Leavitt path algebra $L_\C(E)$, where the underlying field is the complex numbers $\C$.  This Leavitt path algebra has a natural $*$-algebra structure, and in fact it is isomorphic to a dense $*$-subalgebra of the graph $C^*$-algebra $C^*(E)$.  Whereas most of the existing literature has focused on the algebra structure of $L_K(E)$, we will examine $L_\C(E)$ as a $*$-algebra, an algebra, a $*$-ring, and a ring.  What we find is that the ring structure of $L_\C(E)$ emerges as important in determining the $C^*$-algebra structure of $C^*(E)$.  In fact, we make two conjectures in this regard; the Isomorphism Conjecture for Graph Algebras: \emph{If $E$ and $F$ are graphs, then $L_\C(E) \cong L_\C(F)$ (as rings) implies that $C^*(E) \cong C^*(F)$ (as $*$-algebras)}; and the Morita Equivalence Conjecture for Graph Algebras: \emph{If $E$ and $F$ are graphs, then $L_\C(E)$ is Morita equivalent to $L_\C(F)$ implies that $C^*(E)$ is strongly Morita equivalent to $C^*(F)$.}  We are able to verify these conjectures in two important special cases: (1) when the graphs have no cycles (or equivalently, when the $C^*$-algebras are AF and the algebras are ultramatricial); and (2) when the graphs are row-finite and the associated algebras are simple.

This paper is organized as follows.  We begin with some preliminaries in Section~\ref{preliminaries} to establish notation and basic facts.  In Section~\ref{densesubalgebra} we give a proof of the well-known result that $L_{\C}(E)$ is isomorphic to a dense $*$-subalgebra of $C^*(E)$.  Furthermore, we identify those situations in which $L_{\C}(E)$ is equal to $C^*(E)$.   In Section~\ref{Hom-iso-sec} we show that $*$-homomorphisms between Leavitt path algebras over $\C$ extend to homomorphisms between the associated graph $C^*$-algebras.  As a corollary, we obtain the fact that if $E$ and $F$ are graphs, then $L_\C(E) \cong L_\C(F)$ (as $*$-algebras) implies $C^*(E) \cong C^*(F)$ (as $*$-algebras).  We also examine which isomorphisms between graph $C^*$-algebras can be obtained in this way.  In Section~\ref{alg-and-*-alg-hom-sec} we examine algebra homomorphisms between Leavitt path algebras over $\C$ and show that these do not necessarily extend to homomorphisms between the associated graph $C^*$-algebras.  We obtain necessary and sufficient conditions for an algebra homomorphism between Leavitt path algebras over $\C$ to be an algebra $*$-homomorphism.  We also examine some phenomena that motivate our conjectures.  In Section~\ref{Iso-Prob-sec} we present the Isomorphism Conjecture for Graph Algebras.  In Section~\ref{ultramatricial-sec} we show that the Isomorphism Conjecture is true for graphs with no cycles, and in Section~\ref{simple-sec} we prove that the Isomorphism Conjecture is true whenever the graphs in question are row-finite and the associated algebras (equivalently, the associated $C^*$-algebras) are simple.  In Section~\ref{Mor-Equiv-sec} we state and investigate the Morita Equivalence Conjecture for Graph Algebras.   We conclude in Section~\ref{converses-sec} with some results which provide a converse for the Isomorphism Conjecture in certain special situations.

Part of the beauty of the current investigation is that tools from both the algebraic and analytic sides are brought to bear.  For example, along the way we will use such analytic gems as the Kirchberg-Phillips Classification Theorem and the Brown-Green-Rieffel Theorem; and such algebraic pearls as the Graded Uniqueness Theorem and the Stephenson Theorem on infinite matrix rings.

\bigskip

\noindent \textbf{Notation and Conventions:} In this paper we consider rings, algebras, $*$-algebras, and $C^*$-algebras.  Sometimes we will have objects that are simultaneously in more than one of these classes, and our viewpoint may switch from one class to another.  To make statements precise, the term \emph{ring homomorphism} will always mean a function that is additive and multiplicative; and the term \emph{algebra homomorphism} will mean  a ring homomorphism that is $K$-linear.  A \emph{ring $*$-homomorphism} (respectively, an \emph{algebra $*$-homomorphism}) will mean a ring homomorphism (respectively, an algebra homomorphism) that is $*$-preserving.  Likewise, for two objects $A$ and $B$ we write $A \cong B$ (as rings) to mean there is a ring isomorphism from $A$ to $B$, we write $A \cong B$ (as algebras) to mean there is an algebra isomorphism from $A$ to $B$, we write $A \cong B$ (as $*$-rings) to mean there is a ring $*$-isomorphism from $A$ to $B$, and we write $A \cong B$ (as $*$-algebras) to mean there is an algebra $*$-isomorphism from $A$ to $B$.

In addition, for a topological space $X$ that is locally compact and Hausdorff, we let $C(X)$ denote the $C^*$-algebra consisting of continuous complex-valued functions on $X$.  We also let $\K = \K (H)$ denote the $C^*$-algebra of compact operators on a separable infinite-dimensional Hilbert space $H$.  Given a set of elements $S$ in an algebra or $C^*$-algebra, we let $\algspan S$ denote the \emph{algebraic span} of $S$ consisting of all finite linear combinations of elements of $S$.  Given a set of elements $S$ in a $C^*$-algebra, we let $\clspan S$ denote the \emph{closed linear span} of $S$, which is equal to the closure of $\algspan S$.  We say that an algebra is \emph{ultramatricial} if it is the direct limit of a collection of finite-dimensional subalgebras.  In the literature, ultramatricial is sometimes also called locally matricial.  We say that a $C^*$-algebra is an \emph{AF-algebra} if it is the direct limit of a sequence of finite-dimensional $C^*$-algebras (or equivalently, if $A$ is the closure of the increasing union of a countable collection of finite-dimensional algebras).  The abbreviation ``AF" stands for ``Approximately Finite".

Since we hope that this paper will be of interest to both functional analysts and algebraists, we do our best to make the exposition clear and accessible to members from either group.  We give background and reference as much as we can, and we try to be explicit when we use well-known results from functional analysis or algebra that members of the other group may find unfamiliar.

\bigskip

\noindent \textbf{Acknowledgements:} The authors  thank Jutta Hausen and K. M. Rangaswamy for discussions which helped the authors to realize that Lemma~\ref{free-divisible-lem} would be useful in the proof of Theorem~\ref{K-theory-implications-thm}.  The authors  thank the University of Houston for sponsoring a visit of the first author, during which some of the preliminary work on this paper was conceived.    The authors are extremely grateful to the referee for a robust and helpful report, and for suggesting Example \ref{anotherexampleofnotalgebraic}.

\bigskip

\section{Graph algebra preliminaries}\label{preliminaries}

\begin{definition}

A \emph{$*$-ring} (also called an \emph{involutory ring}) is a ring $R$ together with an an involution $*$ satisfying
\begin{itemize}
\item[(1)] $(a^*)^* = a$ for all $a \in R$,
\item[(2)] $(a + b)^* = a^* + b^*$ for all $a,b \in R$,
\item[(3)] $(ab)^* = b^*a^*$ for all $a, b \in R$.
\end{itemize}
A $*$-algebra is an algebra $A$ over the complex numbers with an involution $*$ that is also antilinear; i.e., a ring involution that also satisfies
\begin{itemize}
\item[(4)] $(\lambda a)^* = \overline{\lambda} a^*$ for all $\lambda \in \C$ and $a \in A$.
\end{itemize}

\end{definition}

\begin{definition}
Suppose that $R$ is a $*$-ring (or a $*$-algebra) with involution $*$.  We
call an element $p \in R$ a \emph{projection} if $p=p^2=p^*$.  If
$p$ and $q$ are projections we say that $p$ and $q$ are
\emph{orthogonal} if $pq = 0$, and we say that $p \leq q$ if
$qp=p$.  We call an element $s \in R$  a \emph{partial isometry}
if $ss^*s=s$ and $s^*ss^*=s^*$.  (Note that in this case $ss^*$ and
$s^*s$ are projections.)  We say two partial isometries $s$ and
$t$ have \emph{orthogonal ranges} if $s^*t = 0$.
\end{definition}

\begin{definition}
A \emph{graph} $(E^0, E^1, r, s)$ consists of a countable set $E^0$ of vertices, a countable set $E^1$ of edges, and maps $r : E^1 \to E^0$ and $s : E^1 \to E^0$ identifying the range and source of each edge.  \end{definition}

\begin{remark} \label{countable-hypothesis-rem}
We require our graphs to be countable for two reasons: First, we wish to apply the Kirchberg-Phillips Classification Theorem to $C^*$-al\-ge\-bras associated to graphs.  In order for the hypothesis of separability to be satisfied, we need the graph to be countable so that the $C^*$-algebra has a countable approximate unit.  Second, we wish to apply Proposition~\ref{Mor-equiv-iso-mat-prop} to Leavitt path algebras of graphs, and we need the countability of the graph to ensure that the algebra has a countable set of {\it enough idempotents} (see Remark~\ref{enough-idempotents-rem} and Definition~\ref{enough-idempotents-def}).
\end{remark}

\begin{definition}
Let $E := (E^0, E^1, r, s)$ be a graph. We say that a vertex $v
\in E^0$ is a \emph{sink} if $s^{-1}(v) = \emptyset$, and we say
that a vertex $v \in E^0$ is an \emph{infinite emitter} if
$|s^{-1}(v)| = \infty$.  A \emph{singular vertex} is a vertex that
is either a sink or an infinite emitter, and we denote the set of
singular vertices by $E^0_\textnormal{sing}$.  We also let
$E^0_\textnormal{reg} := E^0 \setminus E^0_\textnormal{sing}$, and
refer to the elements of $E^0_\textnormal{reg}$ as \emph{regular
vertices}; i.e., a vertex $v \in E^0$ is a regular vertex if and
only if $0 < |s^{-1}(v)| < \infty$.  A graph is \emph{row-finite}
if it has no infinite emitters.  A graph is \emph{finite} if both
sets $E^0$ and $E^1$ are finite (or equivalently, when $E^0$ is
finite and $E$ is row-finite).
\end{definition}

\begin{definition}
If $E$ is a graph, a \emph{path} is a sequence $\alpha := e_1 e_2
\ldots e_n$ of edges with $r(e_i) = s(e_{i+1})$ for $1 \leq i \leq
n-1$.  We say the path $\alpha$ has \emph{length} $| \alpha| :=n$,
and we let $E^n$ denote the set of paths of length $n$.  We
consider the vertices in $E^0$ to be paths of length zero.  We
also let $E^* := \bigcup_{n=0}^\infty E^n$ denote the paths of
finite length in $E$, and we extend the maps $r$ and $s$ to $E^*$
as follows: For $\alpha = e_1 e_2 \ldots e_n \in E^n$ with $n\geq
1$, we set $r(\alpha) = r(e_n)$ and $s(\alpha) = s(e_1)$; for
$\alpha = v \in E^0$, we set $r(v) = v = s(v)$.  A \emph{cycle} is a path $\alpha=e_1 e_2 \ldots e_n$ with $r(\alpha) = s(\alpha)$ and $s(e_i)\neq s(e_j)$ for all $1\leq i\neq j \leq n$ .  If $\alpha = e_1e_2 \ldots e_n$ is a cycle, an \emph{exit} for $\alpha$ is an edge $f \in E^1$ such that $s(f) = s(e_i)$ and $f \neq e_i$ for some $i$.  We say that a graph satisfies Condition~(L) if every cycle has an exit.   Note that a graph with no cycles vacuously satisfies Condition~(L).  We denote by $E^\infty$ the set of infinite paths $\gamma= \gamma_1\gamma_2 \ldots$ of the graph $E$, and we say that $E$ is \emph{cofinal} if for every $v \in E^0$ and every $\gamma \in E^\infty$ there is a vertex $w$ on the path $\gamma$ such that there is a finite path from $v$ to $w$.
\end{definition}

\begin{definition} \label{graph-C*-def}
If $E$ is a graph, the \emph{graph $C^*$-algebra} $C^*(E)$ is the universal
$C^*$-algebra generated by mutually orthogonal projections $\{ p_v
: v \in E^0 \}$ and partial isometries with mutually orthogonal
ranges $\{ s_e : e \in E^1 \}$ satisfying
\begin{enumerate}
\item $s_e^* s_e = p_{r(e)}$ \quad  for all $e \in E^1$
\item $p_v = \sum_{\{ e \in E^1 : s(e) = v \}} s_es_e^* $ \quad for all $v \in E^0_\textnormal{reg}$
\item $s_es_e^* \leq p_{s(e)}$ \quad for all $e \in E^1$.
\end{enumerate}
(Universal in this definition means that if $A$ is any $C^*$-algebra containing a family of mutually orthogonal projections $\{ q_v : v \in E^0 \}$ and a family of partial isometries with mutually orthogonal ranges $\{ t_e : e \in E^1 \}$ satisfying Conditions (1)--(3) above, then there exists a unique algebra $*$-homomorphism $\phi : C^*(E) \to A$ with $\phi(p_v) = q_v$ for all $v \in E^0$ and $\phi(s_e) = t_e$ for all $e \in E^1$.)  We mention that when $E$ is row-finite, Condition (2) implies Condition (3).
\end{definition}

\begin{definition} We call Conditions (1)--(3) in Definition~\ref{graph-C*-def} the \emph{Cuntz-Krieger relations}.  For any $*$-ring $R$, a collection of mutually orthogonal
projections $\{P_v : v \in E^0 \}$ and partial isometries with
mutually orthogonal ranges $\{ S_e : e \in E^1\}$ in $R$ which
satisfy (1)--(3) is called a \emph{Cuntz-Krieger $E$-family}
in $R$ .
\end{definition}

\begin{definition} \label{LPA-lin-invo-def}
Let $E$ be a graph, and let $K$ be a field. We let $(E^1)^*$
denote the set of formal symbols $\{ e^* : e \in E^1 \}$, and for
$\alpha = e_1 \ldots e_n \in E^n$ we define $\alpha^* := e_n^*
e_{n-1}^* \ldots e_1^*$.  We also define $v^* = v$ for all $v \in
E^0$.  We call the elements of $E^1$ \emph{real edges} and the
elements of $(E^1)^*$ \emph{ghost edges}.  The \emph{Leavitt path
algebra of $E$ with coefficients in $K$}, denoted $L_K(E)$,  is
the free associative $K$-algebra generated by a set $\{v : v \in
E^0 \}$ of pairwise orthogonal idempotents, together with a set
$\{e, e^* : e \in E^1\}$ of elements, modulo the ideal generated
by the following relations:
\begin{enumerate}
\item $s(e)e = er(e) =e$ for all $e \in E^1$
\item $r(e)e^* = e^* s(e) = e^*$ for all $e \in E^1$
\item $e^*f = \delta_{e,f} \, r(e)$ for all $e, f \in E^1$
\item $v = \displaystyle \sum_{\{e \in E^1 : s(e) = v \}} ee^*$ whenever $v \in E^0_\reg$.
\end{enumerate}
For any Leavitt path algebra $L_K(E)$, there is a $K$-linear
involution $x \mapsto x^\wedge$ with $e^\wedge = e^*$ and
$v^\wedge = v$ for all $e \in E^1$, $v \in E^0$.  Hence for a general element we have
$\left( \sum \lambda_i \alpha_i \beta_i^* \right)^\wedge = \sum \lambda_i
\beta_i \alpha_i^*$.   In addition, $L_K(E)$ is $\Z$-graded, with the grading induced by requiring that $\alpha \beta^*$ is in the homogeneous component of degree $|\alpha| - |\beta|$.
\end{definition}

\begin{remark} \label{enough-idempotents-rem}
The Leavitt path algebra $L_K(E)$ is unital if and only if  $E^0$ is finite; in this case,  $1 = \sum_{v\in E^0}v$.   When $E^0$ is infinite then $L_K(E)$ contains a set of enough idempotents, consisting of finite sums of distinct elements of $E^0$.
\end{remark}

\begin{remark} \label{LPA-univ-rem}
Leavitt path algebras also have a universal property:  If $A$ is a $K$-algebra, and there is a set of elements $\{ a_v, b_e, c_e : v \in E^0, e \in E^1 \}$ satisfying
\begin{enumerate}
\item the $a_v$'s are pairwise orthogonal idempotents
\item $a_{s(e)}b_e = b_ea_{r(e)} =b_e$ for all $e \in E^1$
\item $a_{r(e)}c_e = c_e a_{s(e)} = c_e$ for all $e \in E^1$
\item $c_eb_f = \delta_{e,f} \, a_{r(e)}$ for all $e, f \in E^1$
\item $a_v = \displaystyle \sum_{\{e \in E^1 : s(e) = v \}} b_ec_e$ whenever $v \in E^0_\reg$,
\end{enumerate}
then there exists a unique algebra homomorphism $\phi : L_K(E) \to A$ satisfying $\phi (v) = a_v$ for all $v \in E^0$, $\phi(e) = b_e$ for all $e \in E^1$, and $\phi(e^*) = c_e$ for all $e \in E^1$.  We will call a collection $\{a_v, b_e, c_e: v \in E^0, e \in E^1 \}$ satisfying (1)--(5) above a \emph{Leavitt $E$-family}.
\end{remark}

Throughout the sequel we will be investigating the relationship
between various graph $C^*$-algebras and Leavitt path algebras. We will use the term \emph{graph algebra} to refer to either a graph $C^*$-algebra or a Leavitt path algebra.

Much of our analysis in this paper will involve Leavitt path algebras over the field $\C$ of complex numbers.  The Leavitt path algebra $\LC (E)$ is special in several regards.
First, in addition to the linear involution $x \mapsto
x^\wedge$ described in Definition~\ref{LPA-lin-invo-def}, there also exists a conjugate linear involution $x
\mapsto x^*$ given by $(\sum \lambda_i \alpha_i \beta_i^*)^* =
\sum \  \overline{\lambda_i} \ \beta_i \alpha_i^*$.  Note that
$v^* = v$ and $(e)^* = e^*$ for all $v \in E^0, e \in E^1$.  With
this involution, $\LC(E)$ is a complex $*$-algebra.  Furthermore, in addition to the universal property of $L_\C(E)$ in the category of algebras and algebra homomorphisms (described in Remark~\ref{LPA-univ-rem}),
$\LC(E)$ also has a universal property in the category of complex $*$-algebras:  If $A$ is a complex $*$-algebra and $\{a_v, b_e : v \in E^0, e \in E^1 \} \subseteq A$ is a set of elements satisfying
\begin{enumerate}
\item the $a_v$'s are pairwise orthogonal and $a_v=a_v^2=a_v^*$ for all $v \in E^0$
\item $a_{s(e)}b_e = b_ea_{r(e)} =b_e$ for all $e \in E^1$
\item $b_e^*b_f = \delta_{e,f} \, a_{r(e)}$ for all $e, f \in E^1$
\item $a_v = \displaystyle \sum_{\{e \in E^1 : s(e) = v \}} b_eb_e^*$ whenever $v \in E^0_\reg$,
\end{enumerate}
then there exists a unique algebra $*$-homomorphism $\phi : L_\C(E) \to A$ satisfying $\phi (v) = a_v$ for all $v \in E^0$ and $\phi(e) = b_e$ for all $e \in E^1$.

\begin{remark}
We see that for a given graph $E$, the $C^*$-algebra $C^*(E)$ is universal for Cuntz-Krieger $E$-families in the category of $C^*$-algebras and algebra $*$-homomorphisms, the $K$-algebra $L_K(E)$ is universal for Leavitt $E$-families in the category of $K$-algebras and $K$-algebra homomorphisms, and $L_\C(E)$ is universal for Leavitt $E$-families in the category of complex $*$-algebras and algebra $*$-homomorphisms.  (We note that in these categories, morphisms are not required to preserve identity elements.)
\end{remark}

\begin{definition} \label{g-i-def}
If $E$ is a graph and $\{p_v, s_e \}$ is a Cuntz-Krieger $E$-family generating $C^*(E)$, then for any $z \in \T$ (the complex numbers having norm 1) we see that $\{p_v, zs_e \}$ is a Cuntz-Krieger $E$-family in $C^*(E)$.  By the universal property of $C^*(E)$ there exists an algebra $*$-homomorphism $\gamma_z : C^*(E) \to C^*(E)$ with $\gamma_z(p_v) = p_v$ for all $v \in E^0$ and $\gamma_z(s_e) = zs_e$ for all $e \in E^1$.  Since $\gamma_{\overline{z}}$ is an inverse for $\gamma_z$ we have that $\gamma_z$ is a $*$-automorphism.  Thus we obtain an action $\gamma : \T \to \operatorname{Aut} C^*(E)$ given by $z \mapsto \gamma_z$.  We call $\gamma$ the \emph{gauge action} on $C^*(E)$ and for any $z \in \T$ we refer to $\gamma_z$ as the \emph{gauge $*$-automorphism determined by $z$}.

Likewise, if $\{p_v, s_e\}$ is a generating Leavitt $E$-family in $L_{\C}(E)$, then for any $z \in \T$ we may use the universal property of $L_\C(E)$ to obtain an algebra $*$-automorphism $\gamma_z : L_\C(E) \to L_\C(E)$ with $\gamma_z(v) = p_v$ for all $v \in E^0$ and $\gamma_z(e) = zs_e$ for all $e \in E^1$.

In analogy with the graph $C^*$-algebras, if $E$ is a graph and  $\{v, e, e^*\}$ is a Leavitt $E$-family generating $L_K(E)$, then for any $a \in K^\times$ (here $K^\times$ denotes the invertible elements in the field $K$) we see that $\{v, ae, a^{-1}e^* \}$ is a Leavitt $E$-family in $L_K(E)$.  By the universal property of $L_K(E)$ there exists a $K$-algebra homomorphism $\gamma_a : L_K(E) \to L_K(E)$ with $\gamma_a(v) = v$ for all $v \in E^0$ and $\gamma_a(e) = ae$ and $\gamma_a(e^*) = a^{-1}e^*$ for all $e \in E^1$.  Since $\gamma_{a^{-1}}$ is an inverse for $\gamma_a$ we have that $\gamma_a$ is an automorphism.  Thus we obtain an action $\gamma : K^\times \to \operatorname{Aut} L_K(E)$ given by $a \mapsto \gamma_a$.  We call $\gamma$ the \emph{scaling action} on $L_K(E)$ and for any $a \in K^\times$ we refer to $\gamma_a$ as the \emph{scaling automorphism determined by $a$}.

\end{definition}

We close this section by reminding the reader of two
fundamental structures.  We let $R_n$ denote the ``rose with $n$
petals" graph, namely, the graph with one vertex $v$ and edges
$e_1, \ldots ,e_n$ each beginning and ending at $v$.
\begin{example}
For any $n\geq 2$ and any field $K$, the \emph{Leavitt K-algebra of
order n}, denoted $L_K(1,n)$ or often simply by $L_n$, is the free
associative $K$-algebra in the $2n$ variables
$\{X_1,...,X_n,Y_1,...,Y_n\}$, modulo the relations $X_iY_j =
\delta_{i,j}1_K$ (for $1\leq i,j \leq n$) and $\sum_{i=1}^n Y_iX_i
= 1_K$ (see \cite{L}).   Equivalently, $L_K(1,n) =
L_K(R_n)$, under the correspondence $e_i \mapsto Y_i$ and
$e_i^* \mapsto X_i$ .

For any $n\geq 2$, the \emph{Cuntz algebra of order n}, denoted
$\mathcal{O}_n$, is the unital $C^*$-algebra generated by $n$ partial isometries $S_1, \ldots, S_n$ satisfying $1 = \sum_{i=1}^n S_iS_i^*$.  (This definition is independent of the choice of partial isometries; see \cite{C}.)  In addition, $\mathcal{O}_n = C^*(R_n)$, under the correspondence $s_{e_i} \mapsto S_i$.
\end{example}

\section{Viewing $\LC(E)$ as a dense $*$-subalgebra of $C^*(E)$ }\label{densesubalgebra}

Let $E = (E^0, E^1, r,s )$ be a graph.  Since the generators of
$\LC(E)$ satisfy the same relations as the generators of $C^*(E)$,
people will often nonchalantly say that $\LC(E)$ sits as a dense
$*$-subalgebra of $C^*(E)$.  However, this is not immediately
obvious and there are some subtleties to be aware of.  Since the
elements $\{ s_e, p_v : e \in E^1, v \in E^0 \}$ satisfy the
Cuntz-Krieger relations, the universal property of $\LC(E)$ gives
us an algebra homomorphism $\iota_E : \LC(E) \to C^*(E)$ with $\iota_E(e) =
s_e$ and $\iota_E(v) = p_v$.  Thus we have a homomorphic copy of
$\LC(E)$ inside $C^*(E)$.  To see that $\iota_E$ is injective, and
thus that $\iota_E (\LC(E))$ is isomorphic to $\LC(E)$, one needs
more than just the universal property.  Indeed, this fact relies
on the Graded Uniqueness Theorem, which is a fairly deep result.
We make this precise and give a proof of the injectivity of
$\iota_E$ in the following proposition.

\begin{proposition}[{\cite[Theorem~7.3]{Tom10}}] \label{embedding-prop}
Let $E = (E^0, E^1, r,s )$ be a graph.  Then there exists an
injective algebra $*$-homomorphism $\iota_E : \LC(E) \to C^*(E)$ with
$\iota_E (v) = p_v$ and $\iota_E (e) = s_e$ for all $v \in E^0$
and $e \in E^1$.  Consequently, $\LC(E)$ is canonically isomorphic
to the dense $*$-subalgebra $$\iota_E (\LC(E)) =
\operatorname{span} \{ s_\alpha s_\beta^* : \alpha, \beta \in E^*
\text{ and } r(\alpha) = r(\beta)\}$$ of $C^*(E)$.
\end{proposition}

\begin{proof}
Using the universal property of $\LC(E)$ we obtain an algebra
$*$-ho\-mo\-morph\-ism $\iota_E : \LC(E) \to C^*(E)$ with $\iota_E (v) =
p_v$ and $\iota_E (e) = s_e$.  In addition, the universal property
of $C^*(E)$ implies that there exists a gauge action $\gamma : \T
\to \operatorname{Aut} (C^*(E))$ with $\gamma_z(p_v) = p_v$ for
all $v \in E^0$ and $\gamma_z(s_e) = zs_e$ for all $e \in E^1$.  (See Remark~\ref{g-i-def} for details.) Set
$$\mathcal{A} := \iota_E(\LC(E)) = \operatorname{span} \{ s_\alpha
s_\beta^* : \alpha, \beta \in E^* \text{ and } r(\alpha) =
r(\beta)\}.$$  For $n \in \Z$ we may then define $\mathcal{A}_n :=
\{ a \in \mathcal{A} : \int_\T z^{-n} \gamma_z(a) \, dz = a \}$,
where the integration $dz$ is done with respect to normalized Haar
measure on $\T$.  (For details on what this ``$C^*$-algebra-valued integral over $\T$" means, we refer the reader to \cite[Lemma~3.1]{Rae}.)

We see that for an element $\lambda s_\alpha s_\beta^*$, we have
$$\int_\T \lambda s_\alpha s_\beta^* \, dz = \begin{cases} \lambda
s_\alpha s_\beta^* & \text{ if $|\alpha| - |\beta| = n$} \\ 0 &
\text{ otherwise.} \end{cases}$$  Thus an element $x :=
\sum_{k=1}^N \lambda_k s_{\alpha_k} s_{\beta_k}^* \in \mathcal{A}$
is in $\mathcal{A}_n$ if and only if $|\alpha_k| - |\beta_k| = n$
for all $1 \leq k \leq N$.  One can then see that $\mathcal{A} =
\bigoplus_{n \in \Z} \mathcal{A}_n$ as abelian groups.
Furthermore, if $x := \sum_{k=1}^M \lambda_k s_{\alpha_k}
s_{\beta_k}^* \in \mathcal{A}_m$ and $y := \sum_{l=1}^N \kappa_l
s_{\gamma_l} s_{\delta_l}^* \in \mathcal{A}_n$, we have that $$xy
= \sum_{k,l} \eta_{k,l} s_{\mu_{k,l}} s_{\nu_{k,l}}^*,$$ for some $\eta_{k,l} \in \mathbb{C}$ and with
$|\mu_{k,l}| - |\nu_{k,l}| = |\alpha_k| - |\beta_k| + |\gamma_l| -
|\delta_l| = m + n$.  Thus $xy \in \mathcal{A}_{m+n}$, and
$\mathcal{A}$ is $\Z$-graded.  Because $\iota_E(v) = p_v \in
\mathcal{A}_0$, $\iota_E(e) = s_e \in \mathcal{A}_1$, and
$\iota_E(e^*) = s_e^* \in \mathcal{A}_{-1}$, we see that $\iota_E$
is a graded algebra homomorphism.  Because we also have $\iota_E(v) = p_v
\neq 0$ for all $v \in E^0$, it follows from the Graded Uniqueness
Theorem for Leavitt path algebras (see \cite[Theorem~4.8]{Tom10}) that $\iota_E$ is injective.
\end{proof}

\begin{remark}
In view of Proposition~\ref{embedding-prop}, whenever we have a
graph $E$ we may identify $\LC(E)$ with a dense $*$-subalgebra of
$C^*(E)$ via the embedding $\iota_E : \LC(E) \to C^*(E)$.  (In
particular, for each $n\geq 2$ we may view the Leavitt algebra
$\LC(1,n)$ as a dense $*$-subalgebra of the Cuntz algebra
$\mathcal{O}_n$.)  Because of this, we will often write $p_v$, $s_e$, and $s_e^*$ for the generators of $L_\C(E)$, rather than using the notation $v$, $e$, and $e^*$ common for Leavitt path algebras.  This helps us to view $L_\C(E)$ as a $*$-subalgebra of $C^*(E)$, and to identify the respective generators.

In addition, we may consider the norm on $L_\C(E)$ obtained by restricting the norm on $C^*(E)$.  We will, without comment, make reference to this norm throughout the sequel, and when we write $\| x \|$ for $x \in L_\C(E)$, we of course mean the norm of $x$ when viewed as an element in $C^*(E)$.  Note that this norm on $L_\C(E)$ is the restriction of a $C^*$-norm and therefore satisfies $\| x^* x \| = \| x \|^2$ and $\| x^* \| = \|x \|$ for all $x \in L_\C(E)$.
\end{remark}

We now consider when $L_\C(E)$ is the same as $C^*(E)$.  It follows from \cite[Proposition~3.5]{APM} and \cite[Corollary~2.3]{KPR} that if $E$ is a finite graph with no cycles, then $C^*(E) \cong L_\C(E) \cong \bigoplus_{i=1}^k M_{n(v_i)} (\C)$, where $v_1, \ldots, v_k$ are the sinks of $E$ and $n(v_i)$ is the number of paths in $E$ ending at $v_i$ (including the trivial path $v_i$ itself).  In this particular case we have that the map $\iota_E : \LC(E) \to C^*(E)$ is
surjective. In Proposition~\ref{agree-finite-no-cycles} we show that the surjectivity of $\iota_E$ occurs precisely in this situation.

\begin{lemma} \label{ring-*-hom-inj-isometric}
If $A$ and $B$ are unital $C^*$-algebras and $\phi : A \to B$ is a unital ring $*$-homomorphism (n.b.~$\phi$ is not necessarily $\C$-linear) that is injective, then $\| \phi (a) \| = \| a \|$ for all $a \in A$.
\end{lemma}

\noindent The above lemma follows from \cite[Corollary~2.8]{Tom11}, and a self-contained proof of the result can be found in \cite{Tom11}.



\begin{proposition} \label{agree-finite-no-cycles}
If $E$ is a graph, then the following are equivalent:
\begin{enumerate}
\item $\iota_E : \LC(E) \to C^*(E)$ is surjective;
\item $\LC(E) \cong C^*(E)$ (as $*$-algebras);
\item $\LC(E) \cong C^*(E)$ (as $*$-rings);
\item $E$ is a finite graph with no cycles;
\item $L_\C(E)$ is finite dimensional; and
\item $C^*(E)$ is finite dimensional.
\end{enumerate}
Moreover, when the above hold we have that $\LC(E)$ and $C^*(E)$ are both isomorphic to $M_{n(v_1)} ( \C) \oplus \ldots \oplus M_{n(v_k)} (
\C)$, where $v_1, ..., v_k$ are the sinks of $E$ and $n(v_i)$
is the number of directed paths in $E$ ending at $v_i$ for each
$1\leq i \leq k$.
\end{proposition}

\begin{proof}
\noindent $(2) \implies (3)$. This is trivial since any algebra $*$-isomorphism is a ring $*$-isomorphism.

\smallskip

\noindent $(3) \implies (1)$.  Let $\phi: C^*(E) \to L_\C(E)$ be a ring $*$-isomorphism.  We shall first show that $E$ has a finite number of vertices.  Suppose, for the sake of contradiction that $E$ has an infinite number of vertices, and let $v_1, v_2, \ldots$ be a sequence of distinct vertices in $E$.  Let $p_k := \phi^{-1}(v_k)$ for each $k$.  Since $\phi$ (and hence also $\phi^{-1}$) is a ring $*$-homomorphism, $\{ p_k : k \in \N \}$ is a set of mutually orthogonal projections in $C^*(E)$.  Consider the element $p := \sum_{k=1}^\infty (1/2^k) p_k \in C^*(E)$.  (This infinite sum converges since the $p_{k}$'s are mutually orthogonal projections.)  Let $P := \phi(p) \in L_\C(E)$.  We may write $P = \sum_{j=1}^n \lambda_j \alpha_j \beta_j^*$ for $\alpha_j, \beta_j \in E^*$.  Hence for a large enough $N$ we have that $v_N P = 0$.  But then $$ v_N = \phi(p_N) = \phi(2^N p_N p ) = 2^N \phi(p_N) \phi(p) = 2^N v_N P = 0$$ which is a contradiction.  Hence $E$ has a finite number of vertices, and it follows that $C^*(E)$ and $L_\C(E)$ are unital, and that the ring $*$-isomorphism $\phi : C^*(E) \to L_\C(E)$ is unital.  Consequently, since $L_\C(E)$ is contained in a $C^*$-algebra, Lemma~\ref{ring-*-hom-inj-isometric} implies that $\phi$ is isometric; i.e. $\| \phi(a) \| = \| a \|$ for all $a \in C^*(E)$.  It follows that $L_\C(E)$ is complete in the norm it inherits as a subalgebra of $C^*(E)$: If $\{ a_i \}_{i=1}^\infty \subseteq L_\C(E)$ is a Cauchy sequence in $L_\C(E)$, then $\{ \phi^{-1} (a_i) \}_{i=1}^\infty$ in a Cauchy sequence in $C^*(E)$ and $x = \lim \phi^{-1}(a_i)$ for some $x \in C^*(E)$.  But then the fact that $\phi$ is isometric implies that $\lim a_i = \phi(x) \in L_\C(E)$.  Since $L_\C(E)$ has a $C^*$-norm in which it is complete, $L_\C(E)$ is a $C^*$-algebra.

Because $\LC(E)$ is a $C^*$-algebra containing the Cuntz-Krieger $E$-family $\{ v, e : v \in E^0, e \in E^1 \}$, by the universal property of $C^*(E)$ there exists an algebra $*$-homomorphism $\psi : C^*(E) \to \LC(E)$ such that $\psi(p_v) = v$ and $\psi (s_e) = e$ for all $v \in E^0, e \in E^1$. We then see that $\iota_E \circ \psi$ is the identity on $C^*(E)$ (simply check on generators), and thus $\iota_E$ is surjective.

\smallskip

\noindent $(1) \implies (4)$. Since  $\iota_E : \LC(E) \to C^*(E)$
is surjective, and thus an isomorphism by Proposition
\ref{embedding-prop}, we will identify $\LC(E)$ with $C^*(E)$ and
take the generating Cuntz-Krieger $E$-family for both to be $\{
s_e, p_v : e \in E^1, v \in E^0 \}$. We also have that
$$C^*(E) = \LC(E) = \algspan \{ s_\alpha s_\beta^* : \alpha, \beta
\in E^* \text{ and } r(\alpha) = r(\beta) \}.$$

We shall first show that $E$ has a finite number of vertices.  Suppose
to the contrary that $E$ has an infinite number of vertices, and list these elements as $E^0 = \{ v_1, v_2, \ldots \}$.  Then $\sum_{n=1}^\infty \frac{1}{n^2} p_{v_n}$ converges to an element in $C^*(E)$. On the other hand, since $\LC(E) = C^*(E)$, any element of $C^*(E)$ may be written as a finite $\C$-linear combination of terms $s_\alpha s_\beta^*$, and hence every element of $C^*(E)$ is orthogonal to all but a finite number of projections in the set $\{ p_v : v \in E^0 \}$.  But  $\sum_{n=1}^\infty \frac{1}{n^2} p_{v_n}$ is not orthogonal to any $p_v$ with $v \in E^0$.  Hence we have a contradiction.

Next we shall show that each vertex emits a finite number of edges.  To the contrary, suppose that there exists $v \in E^0$ with $s^{-1}(v)$ infinite.  List the edges that $v$ emits as $s^{-1}(v) = \{ e_1, e_2, e_3, \ldots \}$.  Since the $s_{e_n}s_{e_n}^*$'s are mutually orthogonal projections, the sum $ \sum_{n=1}^\infty \frac{1}{n^2} \, s_{e_n}s_{e_n}^*$ converges to an element  in $C^*(E)$.  Since $C^*(E) = \LC(E)$, we may write this element as a finite $\C$-linear combination of terms $s_\alpha s_\beta^*$.  By grouping terms, we may write $$\sum_{n=1}^\infty \frac{1}{n^2} s_{e_n}s_{e_n}^* = \sum_{k=1}^r \lambda_k s_{\alpha_k} s_{\beta_k}^* + \sum_{k=1}^q \mu_k p_{v_k}$$ for some $\alpha_k, \beta_k \in E^*$, $v_k \in E^0$, and $\lambda_k, \mu_k \in \C$ with the $p_{v_k}$'s mutually orthogonal and with either $|\alpha_k | \geq 1$ or $|\beta_k | \geq 1$ for all $k$.  Since there are an infinite number of elements in $s^{-1}(v)$, we may choose $m$ large enough that $e_m \in s^{-1}(v)$ is not an edge appearing in any $\alpha_k$ or $\beta_k$, and $1/m^2 < |\mu_k|$ for all $1 \leq k \leq q$.  Then
\begin{align*}
\frac{1}{m^2} p_{r(e_m)} &= s_{e_m}^* \left( \sum_{n=1}^\infty \frac{1}{n^2} s_{e_n}s_{e_n}^* \right) s_{e_m} \\
&= s_{e_m}^* \left(  \sum_{k=1}^r \lambda_k s_{\alpha_k} s_{\beta_k}^* + \sum_{k=1}^q \mu_k p_{v_k} \right) s_{e_m} \\
&= \begin{cases} 0 & \text{ if $v_k \neq s(e_m)$ for all $k$} \\ \mu_k p_{r(e_m)} & \text{ if $v_k = s(e_m)$ for some $k$,} \end{cases}
\end{align*}
which is a contradiction.  Thus each vertex in $E$ emits a finite number of edges.

Since $E$ has a finite number of vertices and each vertex emits a finite number of edges, it follows that $E$ is a finite graph.  We will now show that $E$ contains no cycles.  We consider two cases, and show that we are led to a contradiction in both.

\noindent \textsc{Case I:} $E$ contains cycles, and at least one cycle $\mu$ has an exit $e$.

Without loss of generality we may assume that $s(\mu) = s(e) = v$.  Then we have   $p_v > s_\mu s_\mu^*$.  The existence of the exit $e$ for $\mu$ implies that $p_v \neq s_\mu s_\mu^*$, since otherwise $0=s_\mu s_\mu^*s_e = p_vs_e =s_e$, a contradiction. In a similar manner we get that $p_v > s_\mu s_\mu^* > s_{\mu\mu} s_{\mu\mu}^* > s_{\mu\mu\mu} s_{\mu\mu\mu}^* \ldots$, with no equality occurring in the chain.  For each $n \in \N$ let $S_n := s_{\mu \mu \ldots \mu}$ where the $\mu$ appears $n$ times.  Let $P_0 := p_v - S_1 S_1^*$, and for each $n \in \N$ let
$$P_n := S_n S_n^* - S_{n+1} S_{n+1}^*.$$  Then the $P_n$'s are nonzero mutually orthogonal projections in $C^*(E)$.
As such, the sum $\sum_{n=1}^\infty \frac{1}{n^2} \, P_n$ converges to an element  in $C^*(E)$. On the other hand, since $\LC(E) = C^*(E)$, we may write this element as a finite $\C$-linear combination of terms $s_\alpha s_\beta^*$.  Since $E$ is row-finite (in fact, finite) by the previous paragraph, and none of the vertices on the cycle $\mu$ are sinks, we may use the Cuntz-Krieger relation $p_w = \sum_{s(e)=w} s_es_e^*$ as needed at each vertex $w$ of $\mu$ to write the  element  in the following form:
$$\sum_{n=1}^\infty \frac{1}{n^2} P_n= \sum_{k=1}^r \lambda_k s_{\alpha_k} s_{\beta_k}^* + \sum_{k=1}^q \mu_k s_{\gamma_k} s_{\delta_k}^*$$ where for each $k$ either $\alpha_k$ or $\beta_k$ has the form $\mu^j \epsilon$ for some $\epsilon \in E^*$ such that $\epsilon$ is not an initial segment of $\mu$, and for each $k$ we have $\gamma_k$ and $\delta_k$ are powers of $\mu$.  If $\alpha = \mu^j \epsilon$ and $\epsilon$ is not an initial segment of $\mu$, then for $n \geq j+1$ we have $P_n s_\alpha = s_\alpha^* P_n = 0$.  Furthermore, if $\alpha = \mu^i$ and $\beta=\mu^j$, then for $n \geq \max \{i+1, j+1 \}$ we have $$P_n s_\alpha s_\beta^* P_n = \begin{cases} P_n & \text{ if $i=j$} \\ 0 & \text{ if $i \neq j$.} \end{cases}$$  Thus for $N$ sufficiently large, we have that $m \geq N$ implies
\begin{align*}
\frac{1}{m^2} P_m &= P_m \left( \sum_{n=1}^\infty \frac{1}{n^2} P_n \right) P_m \\
&= P_m \left( \sum_{k=1}^r \lambda_k s_{\alpha_k} s_{\beta_k}^* + \sum_{k=1}^q \mu_k s_{\gamma_k} s_{\delta_k}^* \right) P_m \\
&= \sum_{\{ k \, : \, |\gamma_k| = |\delta_k| \} } \mu_k P_{i_k}
\end{align*}
and hence $\frac{1}{m^2} = \sum_{\{ k \, : \, |\gamma_k| = |\delta_k| \} } \mu_k $ for all $m \geq N$, which is a contradiction.

\smallskip

\noindent \textsc{Case II:}  $E$ contains cycles, but no cycle in $E$ has an exit.

To show that this case cannot occur, we let $\mu$ be any cycle in $E$.  Since $\mu$ has no exits, we have $s_\mu s_\mu^* = p_v$ and $s_\mu$ is a unitary.  Let $C_\mu := C^*(s_\mu)$ be the $C^*$-subalgebra of $C^*(E)$ generated by $s_\mu$.  Since $s_\mu$ is a unitary, it follows from spectral theory that $C_\mu \cong C( \sigma (s_\mu))$, where $\sigma (s_\mu)$ denotes the spectrum of $\mu$.  We shall show that
$\sigma (s_\mu) = \T$.  Because $s_\alpha$ is a unitary in $B_v$, it follows that $\sigma (s_\alpha) \subseteq \mathbb{T}$.  In addition, since the spectrum of an element in a $C^*$-algebra is always a nonempty set (see \cite[Theorem~VII.3.6]{Co2}), there exists $w \in \sigma (s_\alpha) \cap \mathbb{T}$. Choose any $x \in \mathbb{T}$, and let $z$ be an element of $\mathbb{T}$ with the property that $z^n = w \overline{x}$, where $n = |\mu|$.  We see that $\{zs_e, p_v : e \in E^1, v \in E^0 \}$ is also a Cuntz-Krieger $E$-family generating $C^*(E)$, and hence by the universal property there exists a $*$-homomorphism $\gamma_z : C^*(E) \to C^*(E)$ with $\gamma_z(s_e) = zs_e$ and $\gamma(p_v) = p_v$.  Moreover, since $\gamma_{\overline{z}}$ is an inverse for $\gamma_z$, we have that $\gamma_z$ is an automorphism.  Since $\gamma_z (s_\mu) = z^n s_\mu$, we see that $\gamma_z$ restricts to an automorphism $\gamma_z : C_\mu \to C_\mu$.    Thus
\begin{align*}
w \in \sigma (s_\mu) & \Longleftrightarrow s_\mu - w 1_{C_\mu} \text{ is not invertible in $C_\mu$} \\
& \Longleftrightarrow s_\mu - w p_v \text{ is not invertible in $C_\mu$} \\
& \Longleftrightarrow \gamma_z(s_\mu - w p_v) \text{ is not invertible in $C_\mu$} \\
& \Longleftrightarrow z^n s_\mu - w p_v \text{ is not invertible in $C_\mu$} \\
& \Longleftrightarrow w \overline{x} s_\mu - w 1_{C_\mu} \text{ is not invertible in $C_\mu$} \\
& \Longleftrightarrow \overline{x} s_\mu - 1_{C_\mu} \text{ is not invertible in $C_\mu$} \\
& \Longleftrightarrow s_\mu - x 1_{C_\mu} \text{ is not invertible in $C_\mu$} \\
& \Longleftrightarrow x \in \sigma (s_\mu).
\end{align*}
Because $x$ was an arbitrary element of $\mathbb{T}$, it follows that $\sigma (s_\alpha) = \mathbb{T}$.

By spectral theory we have that $B_v = C^*(s_\alpha) \cong C (\sigma (s_\alpha)) = C( \mathbb{T})$.  However, if $a \in L_\C(E) \cap C_\mu$, then $a = \sum_{k=1}^n \lambda_k s_{\alpha_k} s_{\beta_k}^*$.  Since $p_v$ is the identity of $C_\mu$ we have $a = p_v a p_v =  \sum_{k=1}^n \lambda_k p_v s_{\alpha_k} s_{\beta_k}^*p_v$ so without loss of generality we may assume that $s(\alpha_k) = r(\alpha_k) = v$ for all $k$.  Since $\mu$ is a cycle based at $v$ and having no exits, it follows that each $\alpha_k$ and $\beta_k$ has the form $\mu \mu \ldots \mu$.  Hence $a$ has the form $a = \sum_{k=1}^n \lambda_k s_{\mu^{m_k}}$ for $m_k \in \Z$.  Thus $C_\mu \cong C(\T)$ and $L_\C(E) \cap C_\mu$ is a  $*$-subalgebra of this $C^*$-algebra isomorphic to $\C[x,x^{-1}]$.  Hence $L_\C(E) \cap C_\mu \neq C_\mu$, which contradicts the fact that $L_\C(E) = C^*(E)$.

It follows from the two cases considered above  that $E$ has no cycles. Thus $E$ is a finite graph with no cycles.

\smallskip

\noindent $(4) \implies (5)$. If $E$ is finite with no cycles, then there are a finite number of paths in $E$.  Since $L_\C(E) = \algspan \{ s_\alpha s_\beta^* : \alpha, \beta \in E^* \}$, we see that $L_\C(E)$ is spanned by a finite set and therefore finite dimensional.

\smallskip

\noindent $(5) \implies (6)$.  Since $L_\C(E)$ is a finite-dimensional space, all norms on $L_\C(E)$ are equivalent and $L_\C(E)$ is closed in any norm.  Since $C^*(E)$ is the closure of $L_\C(E)$, we have that $C^*(E) = L_\C(E)$.

\smallskip

\noindent $(6) \implies (2)$. Since $C^*(E)$ is a finite-dimensional space and $L_\C(E)$ is a subspace, it follows that $L_\C(E)$ is finite dimensional.  For any finite-dimensional space, all norms on this space are equivalent and the space is closed in each norm.  Since $C^*(E)$ is the closure of $L_\C(E)$, it follows that $L_\C(E) = C^*(E)$.

\medskip

Moreover, if any (and hence all) of the above conditions are satisfied, then Condition~(4) together with \cite[Corollary~2.3]{KPR} and \cite[Proposition~3.5]{APM} shows that $\LC(E) \cong C^*(E) \cong M_{n(v_1)} ( \C) \oplus \ldots \oplus M_{n(v_k)} (
\C)$, where $v_1, ..., v_k$ are the sinks of $E$ and $n(v_i)$
is the number of directed paths in $E$ ending at $v_i$ for each
$1\leq i \leq k$.
\end{proof}

\begin{remark}
Here is an alternate (although less straightforward) verification of Case II in the proof of (1) $\Longrightarrow$ (4) in Proposition \ref{agree-finite-no-cycles}, using an argument which directly addresses the relationship between $L_\C(E)$ and $C^*(E)$.  If $E$ is a finite graph which contains cycles, but for which no cycle has an exit, then the stable rank of $C^*(E)$ equals 1 by \cite[Theorem 3.4]{DHS}.  On the other hand, in this same situation, the stable rank of $L_\C(E)$ is greater than 1 by \cite[Theorem 2.8]{AraP}.  Thus $C^*(E)$ is not isomorphic to $L_\C(E)$ in this case, so that the injection $\iota_E: L_\C(E) \to C^*(E)$ cannot be surjective.
\end{remark}

\section{Algebra $*$-homomorphisms of graph algebras} \label{Hom-iso-sec}

In the remainder of this paper we will be concerned with isomorphisms between graph $C^*$-algebras, and between Leavitt path algebras over $\C$.  To see that many of our results are exceptional in the context of general isomorphisms between $C^*$-algebras and dense $*$-subalgebras, we consider a few examples here that show how unwieldy things can be in the general situation.  We will refer to these examples throughout the remainder of our paper to make it clear that specific results for graph algebras are truly special.

\subsection{Examples of dense $*$-subalgebras of $C^*$-algebras}  In general, if $A$ is a $C^*$-algebra
with a dense $*$-subalgebra $A_0$, and $B$ is a $C^*$-algebra with
a dense $*$-subalgebra $B_0$, then there is no relationship
between isomorphisms of $A$ and $B$,  and isomorphisms of $A_0$
and $B_0$.  For instance, there are examples where $A$ and $B$ are
isomorphic, but $A_0$ is not isomorphic to $B_0$.  In particular,
if $\phi : A \to B$ is an algebra $*$-isomorphism, then $\phi$ does not
necessarily restrict to an isomorphism between $A_0$ and $B_0$ --- in
fact, there is no reason the restriction $\phi |_{A_0}$ must even take
values in $B_0$.  Similarly, there are examples where $A_0$ and
$B_0$ are isomorphic (and even $*$-isomorphic), but $A$ is not isomorphic to $B$. Here there
are two things that can go wrong:  (1) If $\phi : A_0 \to B_0$ is
an isomorphism, then $\phi$ may not be bounded with respect to the
norms on $A$ and $B$, and hence does not necessarily extend to a
map from $A$ to $B$; or (2) even when $\phi$ does extend to a map
from $A$ to $B$ this extension may not be bijective.

Let us consider a few examples to see how these phenomena can occur.

\begin{example} \label{dens-*-subalg-Ex-1}
Suppose that $X := [0,1] \subseteq \mathbb{R}$ and that $Y := [0,1] \cup [2,3] \subseteq \mathbb{R}$.  Let $A := C(X)$ and let $A_0$ denote the $*$-algebra of polynomials with complex coefficients viewed as functions on $X$.  Likewise, let $B:= C(Y)$ and let $B_0$ denote the $*$-algebra of polynomials with complex coefficients viewed as functions on $Y$.  Then $A_0$ is a dense $*$-subalgebra of $A$, and $B_0$ is a dense $*$-subalgebra of $B$.  If we let $\phi : A_0 \to B_0$ be the function which takes a polynomial $p(x)$ viewed as a function on $X$ and sends it to the same polynomial $p(x)$ viewed as a function on $Y$, then clearly $\phi$ is an algebra $*$-isomorphism from $A_0$ onto $B_0$.  On the other hand, $A = C(X)$ is not isomorphic to $B = C(Y)$ (as $*$-algebras) since $X$ and $Y$ are not homeomorphic.  In particular, $\phi$ does not extend to an algebra $*$-isomorphism from $A$ to $B$.
\end{example}

\begin{example} \label{dens-*-subalg-Ex-2}
Let $A = B = B_0 = C([0,1])$.  Also let $A_0$ be the $*$-algebra of polynomials with complex coefficients viewed as functions on $[0,1]$.  Then $A_0$ is a dense $*$-subalgebra of $A$, and $B_0$ is a dense $*$-subalgebra of $B$.  However, we see that $A$ is isomorphic to $B$ (as $*$-algebras), while $A_0$ is not isomorphic to $B_0$. (To see this, note that $A_0$ has a countable Hamel basis while $B_0$ does not, so the two are not even isomorphic as vector spaces.)  Thus an algebra $*$-isomorphism between $C^*$-algebras need not restrict to an algebra $*$-isomorphism between dense $*$-subalgebras of the $C^*$-algebras.
\end{example}

\begin{example} \label{dens-*-subalg-Ex-3}
Let $X := [0,1]$ and $Y := [0,2]$.  Also let $A := C(X)$, and let $A_0$ denote the $*$-algebra of polynomials with complex coefficients viewed as functions on $X$.  Likewise, let $B:= C(Y)$ and let $B_0$ denote the $*$-algebra of polynomials with complex coefficients viewed as functions on $Y$.  Then $A_0$ is a dense $*$-subalgebra of $A$, and $B_0$ is a dense $*$-subalgebra of $B$.  Let $\phi : A_0 \to B_0$ be the function which takes a polynomial $p(x)$ viewed as a function on $X$ and sends it to the same polynomial $p(x)$ viewed as a function on $Y$, then clearly $\phi$ is an algebra $*$-isomorphism.  However, we see that $\phi$ is not bounded with respect to the norm on $A_0$ inherited from $A$, and the norm on $B_0$ inherited from $B$.  In particular, if we let $p_n(x) = x^n$, then we see that in $A_0$ we have $\| p_n \| = \sup \{ x^n : x \in [0,1] \} = 1$, while in $B_0$ we have $\| \phi (p_n) \| = \sup \{ x^n : x \in [0,2] \} = 2^n$.  Thus it is possible for an algebra $*$-isomorphism between dense $*$-subalgebras of $C^*$-algebras to be unbounded.  We contrast this with the situation for $C^*$-algebras:  It is well-known that if $\phi : A \to B$ is an algebra $*$-homomorphism between $C^*$-algebras, then $ \| \phi \| \leq 1$ (see \cite[Theorem~2.1.7]{Mur}) and it is also known that if $\psi : A \to B$ is an algebra isomorphism between $C^*$-algebras, then $\psi$ is bounded (see \cite[Exercise \#5, Ch.1, p.14]{Dix}).  \end{example}

\subsection{Extending algebra $*$-homomorphisms of Leavitt path algebras to algebra $*$-homomorphisms of graph C$^*$-algebras}

Here we show that the situation for Leavitt path
algebras and graph $C^*$-algebras is exceptional with regards to
the isomorphism properties described above. In
particular, we now show that if $E$
and $F$ are graphs, then $\LC(E) \cong \LC(F)$ (as $*$-algebras) implies that
$C^*(E) \cong C^*(F)$.

\begin{theorem} \label{isos-extend-thm}
Let $E$ and $F$ be graphs.  If $\phi: \LC(E) \to \LC(F)$ is an algebra $*$-homomorphism, then there exists a unique algebra $*$-homomorphism $\overline{\phi} : C^*(E) \to C^*(F)$ making the diagram
\begin{equation*}
\xymatrix{ C^*(E) \ar[r]^{\overline{\phi}} & C^*(F) \\ \LC(E) \ar[u]^{\iota_E} \ar[r]^\phi & \LC(F) \ar[u]_{\iota_F} }
\end{equation*}
commute. Moreover, if $\phi : \LC(E) \to \LC(F)$ is an algebra $*$-isomorphism, then $\overline{\phi} : C^*(E) \to C^*(F)$ is an algebra $*$-isomorphism.
\end{theorem}

\begin{proof}
Let $\{ s_e, p_v : e\in E^1, v \in E^0\}$ be a generating
Cuntz-Krieger $E$-family in $C^*(E)$, and let $\{ t_e, q_v : e \in
F^1, v \in F^0 \}$ be a generating Cuntz-Krieger $F$-family in
$C^*(F)$.  Given $\phi: \LC(E) \to \LC(F)$, we see that $\{
\iota_F(\phi (e)), \iota_F(\phi(v)) : e \in E^1, v \in E^0\}$  is
a Cuntz-Krieger $E$-family in $C^*(F)$.  Hence, by the universal
property of $C^*(E)$, there exists an algebra $*$-ho\-mo\-morph\-ism $\overline{\phi}
: C^*(E) \to C^*(F)$ with $\overline{\phi} (p_v) =  \iota_F(\phi(v))$ and
$\overline{\phi}(s_e) =  \iota_F(\phi (e))$ for all $v \in E^0, e \in E^1$.
It is easy to see that $\overline{\phi} \circ \iota_E = \iota_F
\circ \phi$, since the maps on either side of this equation agree
on elements of the form $\lambda \alpha \beta^*$ and these generate $\LC(E)$ as a $\C$-algebra.  (Here we are using the fact that $\phi$ is a $*$-homomorphism.)   Furthermore, $\overline{\phi}$ is
unique because any other such algebra $*$-homomorphism would agree with
$\overline{\phi}$ on the generators of $C^*(E)$ and hence be equal
to $\overline{\phi}$.

In addition, if $\phi$ is an algebra $*$-isomorphism, then $$\{
\iota_E(\phi^{-1}(f)), \iota_E(\phi^{-1}(w)): f \in F^1, w \in F^0
\}$$ is a Cuntz-Krieger $F$-family in $C^*(E)$, and hence by the
universal property of $C^*(F)$ there exists an algebra $*$-ho\-mo\-morph\-ism $\rho :
C^*(F) \to C^*(E)$ with $\rho (q_w) =  \iota_E(\phi^{-1}(w))$ and
$\rho(t_f) =  \iota_E(\phi^{-1}(f))$ for all $w \in F^0, f \in F^1$.
One can easily see that $\overline{\phi} \circ \rho =
\textnormal{Id}_{C^*(F)}$ and $\rho \circ \overline{\phi} =
\textnormal{Id}_{C^*(E)}$ by checking on generators.  Thus $\overline{\phi}$
is an algebra $*$-isomorphism.
\end{proof}

\begin{corollary} \label{iso-implication}
Let $E$ and $F$ be graphs.  Then $\LC(E) \cong \LC(F)$ (as $*$-algebras) implies that $C^*(E) \cong C^*(F)$ (as $*$-algebras).
\end{corollary}

We mention that Corollary~\ref{iso-implication} allows us to obtain \cite[Theorem~5.1]{AAP} directly from \cite[Theorem~4.14]{AAP}.

\subsection{$*$-homomorphisms that are algebraic} \label{Alg-hom-iso-sec}

\begin{definition}
If $\phi : C^*(E) \to C^*(F)$ is an algebra $*$-homomorphism, we say $\phi$ is \emph{algebraic} if $\phi (\iota_E(\LC(E))) \subseteq \iota_F(\LC(F))$.  (Note that implicitly this requires a choice of the generators for $C^*(E)$ and $C^*(F)$, since the maps $\iota_E$ and $\iota_F$ depend on these choices.)
\end{definition}

\begin{remark}
Let $E$ and $F$ be graphs, and suppose that $\{ s_e, p_v : e \in
E^1, v \in E^0 \}$ is a chosen generating Cuntz-Krieger $E$-family in
$C^*(E)$ and $\{ t_f, q_w : f \in F^1, w \in F^0 \}$ is a chosen
generating Cuntz-Krieger $F$-family in $C^*(F)$.  If $\phi :
C^*(E) \to C^*(F)$ is an algebra $*$-homomorphism, then $\phi$ is algebraic with respect to these choices if
and only if for each $v \in E^0$ and each $e\in E^1$ we have that
every $\phi(p_v)$ and every $\phi(s_e)$ is  equal to a finite
linear combination of finite products of elements of $\{ t_f, q_w
: f \in F^1, w \in F^0 \}$.
\end{remark}

For any two $*$-algebras $A$ and $B$ we will let $\Hom (A,B)$ denote the set of algebra $*$-homomorphisms from $A$ to $B$ and let $\Iso (A,B)$ denote the set of algebra $*$-isomorphisms from $A$ to $B$. Also, if $E$ and $F$ are graphs we define $\Homalg (C^*(E),  C^*(F))$
to be the subset of $\Hom (C^*(E), C^*(F))$ consisting of the
algebra $*$-homomorphisms from $C^*(E)$ to
$C^*(F)$ that are algebraic (with respect to the sets  $\{ s_e, p_v : e \in
E^1, v \in E^0 \}$ and $\{ t_f, q_w : f \in F^1, w \in F^0 \}$).   Theorem~\ref{isos-extend-thm} shows that there is a map $$\Psi : \Hom (\LC(E), \LC(F)) \to \Hom (C^*(E), C^*(F))$$ given by $\Psi(\phi) = \overline{\phi}$, so the image of $\Psi$ is contained in $\Homalg (C^*(E),  C^*(F))$.  In addition, if we define $\Isoalg (C^*(E),  C^*(F))$ to be the subcollection of $\Iso (C^*(E), C^*(F))$ consisting of the algebra $*$-isomorphisms from $C^*(E)$ to $C^*(F)$ that are algebraic, then
Theorem~\ref{isos-extend-thm} shows that $\Psi$ restricts to a map
$$\Psi| : \Iso(\LC(E), \LC(F)) \to \Iso (C^*(E), C^*(F)),$$ and that the image of $\Psi|$ is contained in $\Isoalg (C^*(E),  C^*(F))$.

In general the map $\Psi|$ is not surjective (and, furthermore,
$\Psi$ is also not surjective).  The following example shows this.

\begin{example}
Let $E$ be the graph
\begin{equation*}
\xymatrix{ v_0 \ar[r]^{e_1} & v_1 \ar[r]^{e_2} & v_2 \ar[r]^{e_3} & v_3
\ar[r]^{e_4} & \cdots\\ }
\end{equation*}
Then $C^*(E) \cong \K(H)$ for a separable infinite-dimensional
Hilbert space $H$.  (To see this, note that if we define $$e_{ij} := \begin{cases} S_i S_{i+1} \ldots S_{j-1} & \text{ if $i < j$} \\ S_iS_i^* & \text{ if $i=j$} \\ S_{i-1}^* S_{i-2}^* \ldots S_j^* & \text{ if $i > j$,} \end{cases}$$
then $\{ e_{ij} \}_{i,j \in \N}$ is an infinite set of matrix units generating $C^*(E)$.)

 Let $\{\xi_1, \xi_2,
\ldots \}$ be an orthonormal basis for $H$.  (We recall for our algebraist readers that this means that the elements in this set are orthonormal and span a dense subset of $H$.)  For $i\geq 0$ we let
$P_{v_i}$ be the projection onto $\algspan \{\xi_i\}$, and for
$i\geq 1$ we let $S_{e_i}$ be the partial isometry with initial
space $\algspan \{\xi_i\}$ and final space $\algspan \{\xi_{i-1}
\}$. Then $\{ S_e, P_v : v \in E^0, e \in E^1 \}$ is a universal
Cuntz-Krieger $E$-family generating $C^*(E) \cong \K (H)$.

Let $\eta_0 := \sum_{n=1}^\infty \frac{1}{2^n} \xi_n \in H$, and extend $\{ \eta_0 \}$ to an orthonormal basis $\{ \eta_0, \eta_1, \ldots \}$ for $H$.  For $i\geq 0$ define $Q_{v_i}$ to be
the projection onto $\algspan \{\eta_i\}$, and for $i\geq 1$ let
$T_{e_i}$ be the partial isometry with initial space $\algspan
\{\eta_i\}$ and final space $\algspan \{\eta_{i-1} \}$. Then $\{
T_e, Q_v : v \in E^0, e \in E^1 \}$ is a Cuntz-Krieger $E$-family
generating $C^*(E) \cong \K (H)$. Hence there exists an algebra
$*$-homomorphism $\phi : C^*(E) \to C^*(E)$ with $\phi(P_v) = Q_v$ and
$\phi(S_e) = T_e$. Since $\{ T_e, Q_v : v \in E^0, e \in E^1 \}$
generates $\K (H)$, we see that $\phi$ is surjective, and since
$\K(H)$ is simple, $\phi$ is injective. Thus $\phi$ is an algebra
$*$-isomorphism.  Furthermore, $\phi (P_{v_0}) \eta_0 = Q_{v_0} \eta_0 = \eta_0=
\sum_{n=1}^\infty \frac{1}{2^n} \xi_n$. But if $T$ is any finite
$\C$-linear combination of products of the elements of $\{ S_e, P_v : v
\in E^0, e \in E^1 \}$, then $T \eta_0$ is equal to a finite
linear combination of $\xi_n$'s. Thus $T \eta_0 \neq \eta_0$, and
it follows that $\phi (P_{v_0})$ is not a finite $\C$-linear
combination of the elements of $\{ S_e, P_v : v \in E^0, e \in E^1
\}$.  Hence $\phi$ is an algebra $*$-isomorphism that is not algebraic with respect to the set $\{ S_e, P_v : v \in E^0, e \in E^1\}$.
\end{example}

\begin{example}\label{anotherexampleofnotalgebraic}
Here is another example of a $*$-isomorphism that is not algebraic, and which has the advantage of
coming from a finite graph. Let $E$ be the graph with one vertex and one edge. Then $\LC(E) =
\C[x, x^{-1} ]$, the algebra of Laurent polynomials, and $C^*(E) = C(\T)$, the algebra of
continuous functions on the unit circle $\T$.  Then the map $\iota_E : \LC(E) \to C^*(E)$ is the obvious inclusion. For every $g \in C(\T)$ such that $|g(z )| = 1$ for all $z \in \T$ , there is a $C^*$ -algebra homomorphism $\phi : \C (\T) \to C (\T)$ such that $\phi$ sends the identity function $z \mapsto z$ on $\T$ to the function $g$. Moreover, $\phi$ is an isomorphism whenever $g$ is injective. There are many such choices of $g$ that are
not in $\C[x , x^{-1} ]$.  (For example, we could choose $g : \T \to \C$ defined by $g(e^{2\pi i t}) = e^{2\pi i t^2}$, which satisfies the hypotheses but is not analytic and hence not in $\C[x , x^{-1} ]$.)
\end{example}

 Despite the fact that an algebra $*$-isomorphism between two graph $C^*$-algebras need not be algebraic
(and therefore need not restrict to an algebra $*$-isomorphism between $\LC(E)$ and $\LC(F)$), there are a certain situations, which we discuss in \S\ref{converses-sec}, when $C^*(E) \cong C^*(F)$ (as $*$-algebras) implies $\LC(E) \cong \LC(F)$ (as $*$-algebras).

\section{Algebra homomorphisms of graph algebras } \label{alg-and-*-alg-hom-sec}

In this section we consider algebra homomorphisms between graph algebras that are not necessarily $*$-preserving.  We saw in Example~\ref{dens-*-subalg-Ex-3} that an algebra isomorphism between dense $*$-subalgebras of $C^*$-algebras need not be bounded.  This can also occur when the dense $*$-subalgebras are  Leavitt path algebras.  Let $E$ be the graph
$ $

$$
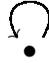
\xymatrix{
\bullet \ar@(ur,ul) \\
}
$$

\noindent consisting of a single vertex and a single edge.
 As noted in Example \ref{anotherexampleofnotalgebraic}, $L_\C(E) \cong \C[x,x^{-1}]$. If we let $p_v :=1$, $s_e:= 2x$, and $s_{e^*} := \frac{1}{2} x^{-1}$, then this collection forms a Leavitt $E$-family, and by the universal property of Leavitt path algebras there is an algebra homomorphism $\psi : L_\C (E) \to L_\C (E)$ with $\psi(1) = 1$, $\psi (x) = 2x$, and $\psi(x^{-1}) = \frac{1}{2} x^{-1}$.  One can see that $\psi$ is one-to-one and onto, and hence $\psi$ is an algebra isomorphism.  In addition, we have that the elements $x^n$ all have norm one, however $\psi(x^n) = 2^nx^n$ has norm $2^n$.  Hence $\psi$ is an algebra isomorphism between Leavitt path algebras that is not bounded.

Nonetheless, the following proposition shows that we are able to scale any algebra automorphism of $\C[x,x^{-1}]$ to obtain an algebra $*$-automorphism.

\begin{proposition}\label{Laurent-scaling}
If $\psi : \C[x,x^{-1}] \to \C[x,x^{-1}]$ is an algebra isomorphism, then there exists $z \in \C^\times$ and an algebra $*$-isomorphism $\phi  : \C[x,x^{-1}] \to \C[x,x^{-1}]$ such that $$\psi = \phi \circ \gamma_z$$ where $\gamma_z$ is the scaling automorphism corresponding to $z$.
\end{proposition}

\begin{proof}
Since $\psi(x) \psi(x^{-1}) = \psi(1) = 1$, we have that $\psi(x)$ is a unit of $\C[x,x^{-1}]$.  However, the only units of $\C[x,x^{-1}]$ are $ax^k$ with $a \in \C^\times$ and $k \in \Z$. (To see this, use the graded structure of this ring and the fact that the ring is a domain.)  Thus we must have $\psi(x) = ax^k$ for some $a \in \C^\times$ and $k \in \Z$, and also $\psi(x^{-1}) = \frac{1}{a} x^{-k}$.  In addition, for any polynomial $p(x, x^{-1})$ in $\C[x,x^{-1}]$, we have that $\psi (p(x, x^{-1})) = p (ax^k, \frac{1}{a} x^{-k})$.  Since $\psi$ is onto, it must be the case that $k = \pm 1$.

If we let $\phi := \psi \circ \gamma_a$, then $\phi$ is an algebra isomorphism (since it is the composition of algebra isomorphisms), and $\phi$ is $*$-preserving (simply check on the generators 1, $x$, and $x^{-1}$).  If we let $z := \frac{1}{a}$, then we see that $\psi = \phi \circ \gamma_z$.
\end{proof}

We are also able to characterize when an algebra homomorphism between complex Leavitt path algebras is an algebra $*$-homomorphism.

\begin{lemma} \label{contract-lem}
Let $U$, $V$, and $P$ be elements in a $C^*$-algebra.  Suppose that $P$ is a projection, and $U$ and $V$ are contractions with $UV = P$.  Then $VP$ is a partial isometry with $(VP)^* (VP) = P$.
\end{lemma}

\begin{proof}
Without loss of generality, we may assume that $U$, $V$, and $P$ are operators on a Hilbert space $H$.  For any $x \in \im P$, the fact that $U$ and $V$ are contractions implies that $$ \| x \| \leq \| P (x) \| = \| UV (x) \| \leq \| V (x) \| \leq \| x \|$$ so that $\| V(x) \| = \| x \|$ and $\| VP (x) \| = \| x \|$.  In addition, if $x \in ( \im P )^\perp$, then $VP (x) = V (Px) = 0$.  Consequently, $VP$ is a partial isometry with initial space $\im P$.  It follows that $(VP)^*(VP) = P$.
\end{proof}

\begin{proposition} \label{*-sufficient-prop}
If $\psi : L_\C(E) \to L_\C(F)$ is an algebra homomorphism between complex Leavitt path algebras, then $\psi$ is an algebra $*$-homomorphism if and only if the following two conditions are satisfied:
\begin{itemize}
\item $\| \psi (s_e) \| \leq 1$ and $\| \psi (s_e^*) \| \leq 1$ for all $e \in E^1$, and
\item $\| \psi (p_v) \| \leq 1$ for each $v \in E^0$ that is a source.
\end{itemize}
\end{proposition}

\begin{proof}
To see that the above conditions are necessary, suppose that $\psi : L_\C(E) \to L_\C(F)$ is an algebra $*$-homomorphism.  Then Theorem~\ref{isos-extend-thm} shows that $\psi$ extends to a $*$-homomorphism from $C^*(E)$ to $C^*(F)$.  Hence $\psi$ must be contractive.

To see that the above conditions are sufficient, let $v \in E^0$ and suppose $v$ is not a source.  Then there exists $e \in E^1$ such that $r(e) = v$.  Since $p_v$ is an idempotent and $\psi$ is an algebra homomorphism, we have that $\psi( p_v)$ is an idempotent.  Furthermore, $\| \psi (p_v) \| = \| \psi ( p_{r(e)} ) \| = \| \psi (s_e^*) \psi (s_e) \|  \leq \| \psi(s_e^*) \| \| \psi (s_e) \| \leq 1$ by hypothesis.  Thus $\psi(p_v)$ is a contractive idempotent, and hence a projection \cite[Proposition~3.3]{Co2}.    Likewise, if $v \in E^0$ is a source, then $\psi(p_v)$ is an idempotent, which is contractive by hypothesis, and hence a projection.  We therefore have that $\psi (p_v)$ is a projection for all $v \in E^0$.

Fix $e \in E^1$.  We have that $s_es_e^*$ is an idempotent, and hence $\psi (s_es_e^*)$ is also an idempotent.  Furthermore, $$\| \psi (s_e s_e^*) \| = \| \psi (s_e) \psi (s_e^*) \| \leq \| \psi (s_e) \| \| \psi (s_e^*) \| \leq 1$$ by hypothesis.  Thus $\psi(s_es_e^*)$ is a contractive idempotent, and hence a projection (again using \cite[Proposition~3.3]{Co2}).   Since $\psi (s_e) \psi(s_e^*) = \psi (s_e s_e^*)$, we may take adjoints of each side of this equation to obtain $$\psi (s_e^*)^* \psi(s_e)^* = \psi (s_es_e^*),$$ and because $\psi(s_e)$ and $\psi(s_e^*)$ are contractions by hypothesis, Lemma~\ref{contract-lem} implies that $\psi(s_e)^* \psi(s_es_e^*) = \psi(s_e)^*$ is a partial isometry with
\begin{equation} \label{p-i-eqn-for-psi}
\psi(s_e) \psi(s_e)^* = \psi(s_e s_e^*).
\end{equation}
Thus
\begin{align*}
\psi(s_e)^* &= [\psi (s_e) \psi( p_{r(e)}) ]^* =  \psi( p_{r(e)} ) \psi(s_e)^* \\
&= \psi(s_e^*) \psi (s_e) \psi(s_e)^* = \psi (s_e^*) \psi (s_e s_e^*) \qquad \text{by (\ref{p-i-eqn-for-psi})} \\
&= \psi (s_e^*).
\end{align*}
Hence $\psi(p_v^*) = \psi(p_v)^*$ for all $v \in E^0$, and $\psi(s_e^*) = \psi(s_e)^*$ for all $e \in E^1$.  Since $\{ s_e, p_v : e \in E^1, v \in E^0 \}$ generates $L_\C(E)$ as a $*$-algebra, it follows that $\psi (x^*) = \psi(x)^*$ for all $x \in L_\C(E)$ and $\psi$ is an algebra $*$-homomorphism.
\end{proof}

\begin{corollary} \label{*-iso-norm-one-cor}
An algebra isomorphism $\psi : L_\C(E) \to L_\C(F)$ is an algebra $*$-isomorphism if and only if the following two conditions are satisfied:
\begin{itemize}
\item $\| \psi (s_e) \| = \| \psi (s_e^*) \| = 1$ for all $e \in E^1$, and
\item $\| \psi (p_v) \| = 1$ for each $v \in E^0$ that is a source.
\end{itemize}
\end{corollary}

\begin{proof}
If $\psi : L_\C(E) \to L_\C(F)$ is an algebra $*$-isomorphism, then by Theorem~\ref{isos-extend-thm} $\psi$ extends to an algebra $*$-isomorphism $\overline{\psi} : C^*(E) \to C^*(F)$ between $C^*$-algebras.  Hence $\psi$ is isometric, and the displayed conditions hold.  The converse follows from Proposition~\ref{*-sufficient-prop}.
\end{proof}

\begin{remark}
It is well known that there are algebra isomorphisms between $C^*$-algebras that are not algebra $*$-isomorphisms.  For example, if $W$ is an invertible operator on a Hilbert space $\mathcal{H}$, then the function $\psi : B(\mathcal{H}) \to B(\mathcal{H})$ given by $\psi (X) = WXW^{-1}$ is an algebra isomorphism that is not in general an algebra $*$-isomorphism.  Despite this fact, it was shown by Gardner in 1965 that if $A$ and $B$ are $C^*$-algebras, then $A \cong B$ (as algebras) if and only if $A \cong B$ (as $*$-algebras) \cite{Gar1, Gar2}.  Furthermore, Gardner showed in \cite[Corollary~4.2]{Gar1} that an algebra isomorphism $\psi : A \to B$ is an algebra $*$-isomorphism if and only if $ \| \psi \| = 1$.  (Compare this with Corollary~\ref{*-iso-norm-one-cor}.)
\end{remark}

In view of Gardner's result that for $C^*$-algebras $A$ and $B$ one has that $A \cong B$ (as algebras) implies $A \cong B$ (as $*$-algebras), it is natural to make the following conjecture for complex Leavitt path algebras.

$ $

\noindent \textbf{Conjecture 1:}  If $E$ and $F$ are graphs, then $\LC(E) \cong \LC(F)$ (as algebras) implies that $\LC(E) \cong \LC(F)$ (as $*$-algebras).

$ $

Along these lines, we can also make the following conjecture.

$ $

\noindent \textbf{Conjecture 2:}   If $E$ and $F$ are graphs, then $\LC(E) \cong \LC(F)$ (as algebras) implies that $C^*(E) \cong C^*(F)$ (as $*$-algebras).

$ $

Note that Corollary~\ref{iso-implication} shows that an affirmative answer to Conjecture~1 implies an affirmative answer to Conjecture~2.

Unfortunately, we are unable to make progress on Conjecture~1, and we are also unable to give a complete answer to Conjecture~2.  We offer here some remarks on how one might approach a solution.  Whenever we have an algebra isomorphism $\psi : L_\C(E) \to L_\C(F)$, then if $\psi$ is bounded it will extend (by continuity) to an algebra isomorphism on the completions $\overline{\psi} : C^*(E) \to C^*(F)$, and an application of Gardner's result shows that $C^*(E) \cong C^*(F)$ (as $*$-algebras).  However, as we have seen at the beginning of this section, isomorphisms between complex Leavitt path algebras need not be bounded, and hence need not be continuous, and therefore need not extend to the completion.  With the Laurent polynomials, we saw in Proposition \ref{Laurent-scaling} that any isomorphism ---although not necessarily bounded--- can be composed with a scaling automorphism to produce a bounded isomorphism.  If a similar result holds for general complex Leavitt path algebras, then using the argument of this paragraph, we would have a positive solution to Conjecture~2.
It is unclear to the authors at this time whether such a result holds.

In the next section we consider a conjecture that is stronger than Conjecture~2; i.e., that $\LC(E) \cong \LC(F)$ (as rings) implies $C^*(E) \cong C^*(F)$ (as $*$-algebras).  This conjecture clearly implies Conjecture~2, and remarkably, even though we are unable to make general progress on Conjecture~2, we are able to answer this stronger conjecture in some special cases.

\section{Ring $*$-homomorphisms, ring homomorphisms,  and the Isomorphism Conjecture for Graph Algebras} \label{Iso-Prob-sec}

Up to this point we have considered algebra homomorphisms and algebra $*$-homomorphisms between complex Leavitt path algebras.  In the remainder of this paper we will focus on the ring structure of complex Leavitt path algebras.

\begin{proposition} \label{ring-*-hom-replace}
If $\psi : L_\C(E) \to L_\C(F)$ is a ring $*$-homomorphism, then there exists an algebra $*$-homomorphism $\phi : L_\C(E) \to L_\C(F)$ with $\phi (p_v) = \psi(p_v)$ and $\phi(s_e) = \psi(s_e)$ for all $v \in E^0, e \in E^1$.
\end{proposition}

\begin{proof}
Note that since $\psi$ is a ring $*$-homomorphism, $\{ \psi(p_v), \psi(s_e) : v \in E^0, e \in E^1 \}$ is a Leavitt $E$-family in $L_\C(F)$.  Thus by the universal property of $L_\C(E)$ there exists an algebra $*$-homomorphism $\phi : L_\C(E) \to L_\C(F)$ with $\phi (p_v) = \psi(p_v)$ and $\phi(s_e) = \psi(s_e)$ for all $v \in E^0, e \in E^1$.
\end{proof}

\begin{remark}
Note that even though the maps $\psi$ and $\phi$ of Proposition~\ref{ring-*-hom-replace} agree on the elements $\{ p_v, s_e : v \in E^0, e \in E^1 \}$, the two maps are not necessarily equal.  This is because the Leavitt $E$-family $\{ p_v, s_e : v \in E^0, e \in E^1 \}$ generates $L_\C(E)$ as a $*$-algebra but not as a $*$-ring.  For example, if we take $E$ to be the graph with one vertex and no edges, then $L_\C(E) \cong \C$ and the generating Leavitt $E$-family is the singleton set $\{ 1 \}$.  If we define $\psi : \C \to \C$ by $\psi (z) = \overline{z}$ and $\phi : \C \to \C$ by $\phi(z) = z$, then $\psi$ is a ring $*$-homomorphism and $\phi$ is an algebra $*$-homomorphism with $\psi (1) = \phi(1)$, but we see that $\psi$ and $\phi$ are not equal.
\end{remark}

\begin{corollary} \label{ring-C*-hom-replace}
If $\psi : L_\C(E) \to L_\C(F)$ is a ring $*$-homomorphism, then there exists an algebra $*$-homomorphism $\phi : C^*(E) \to C^*(F)$ with $\phi (p_v) = \psi(p_v)$ and $\phi(s_e) = \psi(s_e)$ for all $v \in E^0, e \in E^1$.
\end{corollary}

\begin{proof}
Apply Proposition~\ref{ring-*-hom-replace} followed by Theorem~\ref{isos-extend-thm}.
\end{proof}

Corollary~\ref{ring-C*-hom-replace} shows that for any ring $*$-homomorphism $\psi : L_\C(E) \to L_\C(F)$ we may find an algebra $*$-homomorphism $\phi : C^*(E) \to C^*(F)$ that agrees with $\psi$ on the generating Leavitt $E$-family.  However, it is unclear if having $\psi$ a ring $*$-isomorphism implies that $\phi$ is an algebra $*$-isomorphism.  (In fact, both injectivity and surjectivity are in question.) If such a result were obtained it would show that $L_\C(E) \cong L_\C(F)$ (as $*$-rings) implies that $C^*(E) \cong C^*(F)$ (as $*$-algebras).  (Compare this with Theorem~\ref{isos-extend-thm} and Corollary~\ref{iso-implication}.)  Although we are unable to obtain this result, we make the following stronger conjecture.

$ $

\noindent \textbf{The Isomorphism Conjecture for Graph Algebras:}   If $E$ and $F$ are graphs, then $\LC(E) \cong \LC(F)$ (as rings) implies that $C^*(E) \cong C^*(F)$ (as $*$-algebras).

$ $

Although we cannot verify weaker versions of the Isomorphism Conjecture in general (e.g., when our hypothesis is $\LC(E) \cong \LC(F)$ (as $*$-rings) or when our hypothesis is $\LC(E) \cong \LC(F)$ (as algebras)), we are able to show that the Isomorphism Conjecture holds for some very important classes of graph algebras.  We show in \S\ref{ultramatricial-sec} that we have an affirmative answer when the graphs have no cycles (so that the associated algebras are ultramatricial).  We also show in  \S\ref{simple-sec} that there is an affirmative answer when the graph algebras are simple and come from row-finite graphs.  In both cases we accomplish our results by using classification theorems from the Elliott Classification Program for $C^*$-algebras

\section{Isomorphisms of ultramatricial graph algebras} \label{ultramatricial-sec}

In this section we prove that the Isomorphism Conjecture for Graph Algebras has an affirmative answer when the graphs have no cycles (equivalently, the Leavitt path algebras are ultramatricial; equivalently, the graph $C^*$-algebras are AF).  Before we do so, we need to establish some $K$-theory notation, and give a formulation of Elliott's classification for direct limits of semisimple algebras that is useful for our situation.

If $R$ is a ring with unit, $K_0(R)$ is the Grothendieck group of the semigroup of finitely generated projective right $R$-modules under the operation of direct sum.  The group $K_0(R)$ is abelian and consists of expressions $[X]-[Y]$, where $X$ and $Y$ are finitely generated projective right $R$-modules.  Two such expressions $[X]-[Y]$ and $[Z]-[W]$ are equal in $K_0(R)$ if and only if there exists a finitely generated projective right $R$-module $V$ such that $X \oplus W \oplus V \cong Z \oplus Y \oplus V$, and the sum of two expressions $[X]-[Y]$ and $[Z]-[W]$ is equal to $[X \oplus Z] - [Y \oplus W]$.  We write $[X]$ for the expression $[X]-[0]$, and define $$K_0(R)^+ := \{ [X] : X \text{ is a finitely generated projective right $R$-module} \}.$$  One can show $(K_0(R), K_0(R)^+)$ is a preordered group with order unit $[R]$.

Given a unital ring homomorphism $\phi : R \to S$ we make $S$ into a left $R$-module via $\phi$, and use the functor $(-)\otimes_R S$ to map right $R$-modules to right $S$-modules.  This induces a homomorphism $K_0(\phi) : (K_0(R), K_0(R)^+, [R]) \to (K_0(S), K_0(S)^+, [S])$ via $K_0(\phi) ([X]-[Y]) := [X \otimes_R S]-[Y \otimes_R S]$.  Furthermore, $K_0(\phi)$ is an isomorphism whenever $\phi$ is an isomorphism.   The assignment $R \mapsto (K_0(R), K_0(R)^+, [R])$ and $\phi \mapsto K_0(\phi)$ defines a functor from the category of rings with unit to the category of preordered abelian groups with order unit.

If $R$ is a $\C$-algebra, the \emph{unitization} of $R$ is the $\C$-algebra $R^1$ which as a vector space is equal to $R \oplus \C$ and has multiplication defined as $(r, \lambda) (s, \mu) := (r + \mu r + \lambda s, \lambda \mu)$.  There is a natural algebra homomorphism $\nu : R^1 \to \C$ given by $\nu (r, \lambda) := \lambda$, and one sees that the kernel of $\nu$ is isomorphic to $R$.

There is then a positive unital group homomorphism $K_0(\nu) : K_0(R^1) \to K_0(\C)$.  When $R$ is unital we have that $K_0(R) \cong \ker K_0(\nu)$ \cite[Proposition~12.1]{GH}.  When $R$ is nonunital, we define $K_0(R) := \ker K_0(\nu)$.  In this case we also view $\ker K_0(\nu)$ as equipped with the preordered abelian group structure inherited from  $K_0(R^1)$.  To take the place of the order unit, we define the \emph{scale} of $K_0(R)$ to be $\Sigma(R) := \{ x \in K_0(R) : 0 \leq x \leq [R^1] \}.$  When $R$ and $S$ are both unital, a positive group homomorphism $\alpha : K_0(R) \to K_0(S)$ is scale preserving if and only if it is unital.

Given two $\C$-algebras $R$ and $S$ and an algebra homomorphism $\phi : R \to S$ we obtain an algebra homomorphism $\phi^1 : R^1 \to S^1$ given by $\phi^1 (r, \lambda) := (\phi(r), \lambda)$.  The induced map $K_0(\phi^1) : K_0(R^1) \to K_0(S^1)$ restricts to a positive homomorphism from $K_0(R)$ to $K_0(S)$ that preserves scales.  We denote this homomorphism by $K_0(\phi)$.  The assignment $R \mapsto (K_0(R), K_0(R)^+, \Sigma(R))$ and $\phi \mapsto K_0(\phi)$ defines a functor from the category of $\C$-algebras to the category of preordered abelian groups with specified sets of positive elements.  This functor preserves direct limits \cite[Proposition~12.2]{GH}, and so also preserves finite direct sums.  In addition, when $R$ is an ultramatricial $\C$-algebra, $(K_0(R), K_0(R)^+)$ is an ordered abelian group.

If $A$ is a $C^*$-algebra, one may define a topological $K_0$ group $K_0^\textnormal{top} (A)$ as the Grothendieck group of Murray-von Neumann equivalence classes of projections in matrices over $A$ (see \cite{RLL} or \cite{WO} for details).  One also defines a preorder and scale for $K_0^\textnormal{top} (A)$.  When $A \otimes \K$ has an approximate unit consisting of projections (which occurs, for example, if $A$ is an AF-algebra), then
\begin{align*}
K_0^\textnormal{top} (A) &= \{ [p] - [q] : \text{$p$ and $q$ are projections in $A \otimes \K$} \} \\
K_0^\textnormal{top} (A)^+ &= \{ [p]  : \text{$p$ is a projection in $A \otimes \K$} \} \\
\Sigma^\textnormal{top}(A) &= \{ [p]  : \text{$p$ is a projection in $A$} \}.
\end{align*}

For any algebra $*$-homomorphism $\phi : A \to B$ between $C^*$-algebras, we obtain algebra $*$-homomorphisms $\phi_n : {\rm M}_n(A) \to {\rm M}_n(B)$ for each $n \in \N$, where $\phi_n$ is obtained by applying $\phi$ to each entry in the matrix.  We let $\phi_\infty : A \otimes \K \to B \otimes \K$ be the algebra $*$-homomorphism induced on the direct limit.   The algebra $*$-homomorphism $\phi : A \to B$ then induces a group homomorphism $K_0^\textnormal{top} (\phi) : K_0^\textnormal{top}(A) \to K_0^\textnormal{top}(B)$ given by $K_0^\textnormal{top}(\phi) ([p] - [q]) = [\phi_\infty(p)] - [\phi_\infty(q)]$.  This assignment defines a functor from the category of $C^*$-algebras to the category of preordered abelian groups with specified sets of positive elements. This functor preserves direct limits \cite[Theorem~6.3.2]{RLL}, and so also preserves finite direct sums.  In addition, when $A$ is an AF-algebra, $(K_0^\textnormal{top} (A), K_0^\textnormal{top} (A)^+)$ is an ordered abelian group.

It turns out that for a $C^*$-algebra $A$ the algebraic $K_0$-group of $A$ agrees with the topological $K_0$-group of $A$ \cite[Theorem~1.1]{Ros}; that is, $$(K_0(A), K_0(A)^+, \Sigma(A)) = (K_0^\textnormal{top}(A), K_0^\textnormal{top}(A)^+, \Sigma^\textnormal{top}(A)).$$  Therefore, when $A$ is a $C^*$-algebra we will, without ambiguity, simply write $(K_0(A), K_0(A)^+, \Sigma(A))$ for the scaled ordered $K_0$-group.

\begin{lemma} \label{alg-K-iso-op-K}
Let $A$ be an AF $C^*$-algebra, and let $R$ be a dense ultramatricial $*$-subalgebra of $A$.  If $\phi : R \hookrightarrow A$ is the inclusion map, then $K_0(\phi) : K_0(R) \to K_0(A)$ is an isomorphism of scaled ordered groups.
\end{lemma}

\begin{proof}
If $A$ and $R$ are unital, the result is \cite[Proposition~1.5]{GH}.  If one of $A$ or $R$ is nonunital, consider the unitizations $A^1$ and $R^1$, let $\phi^1 : R^1 \to A^1$ be the induced unital algebra homomorphism, let $\nu_R : R^1 \to \C$ and $\nu_A : A^1 \to \C$ be the canonical homomorphisms, and observe that the diagram
$$
\xymatrix{ R^1 \ar[r]^{\phi^1} \ar[rd]_{\nu_R} & A^1 \ar[d]^{\nu_A} \\
& \C}
$$ commutes.  Applying the functor $K_0$ we obtain a commutative diagram
$$
\xymatrix{ K_0(R^1) \ar[r]^{K_0(\phi^1)} \ar[rd]_{K_0(\nu_R)} & K_0(A^1) \ar[d]^{K_0(\nu_A)} \\
& K_0(\C)}.
$$  It follows that $K_0(\phi^1)$ restricts to a group isomorphism between $\ker K_0(\nu_R)$ and $K_0(\nu_A)$ that preserves order and scale.  Thus $K_0(\phi) : K_0(R) \to K_0(A)$ is a group isomorphism of scaled ordered groups.
\end{proof}

The following theorem follows from various results in Elliott's work on direct limits of semisimple finite-dimensional algebras.  The version that we state here can be obtained by piecing together various formulations in the existing literature.

\begin{theorem} \label{AF-equivalences-thm}
Let $A$ and $B$ be AF $C^*$-algebras, let $R$ be a dense ultramatricial $*$-subalgebra of $A$, and let $S$ be a dense ultramatricial $*$-subalgebra of $B$.  Then the following are equivalent:
\begin{enumerate}
\item $A \cong B$ (as $*$-algebras),
\item $R \cong S$ (as $*$-algebras),
\item $R \cong S$ (as algebras),
\item $R \cong S$ (as rings),
\item $(K_0(A), K_0(A)^+, \Sigma(A)) \cong  (K_0(B), K_0(B)^+, \Sigma(B))$, and
\item $(K_0(R), K_0(R)^+, \Sigma(R)) \cong  (K_0(S), K_0(S)^+, \Sigma(S))$.
\end{enumerate}

\noindent Moreover, if $$\alpha : (K_0(A), K_0(A)^+, \Sigma(A)) \to (K_0(B), K_0(B)^+, \Sigma(B))$$ is an isomorphism, then there exists an algebra $*$-isomorphism $\phi :A \to B$ such that $K_0(\phi) = \alpha$.  Likewise, if $$\alpha : (K_0(R), K_0(R)^+, \Sigma(R)) \to (K_0(S), K_0(S)^+, \Sigma(S))$$ is an isomorphism, then there exists an algebra $*$-isomorphism $\phi : R \to S$ such that $K_0(\phi) = \alpha$.
\end{theorem}

\begin{proof}  The result follows from the following equivalences.

$(1) \Longleftrightarrow (5)$ is Elliott's Theorem for AF $C^*$-algebras \cite[Theorem~1.3.3]{Ro5}.

$(5) \Longleftrightarrow (6)$ follows from Lemma~\ref{alg-K-iso-op-K}.

$(6) \Longleftrightarrow (3)$ is a theorem of Elliott \cite[Theorem~12.5]{GH}.

$(3) \Longleftrightarrow (2)$ follows from \cite[Theorem~4.3, Appendix]{Ell2}.

$(3) \Longleftrightarrow (4)$ follows from \cite[Remark 4.4]{Ell2}.

\noindent The fact that isomorphisms on $K$-groups lift to $*$-isomorphisms of the associated algebras follows from  \cite[Theorem~1.3.3]{Ro5}, \cite[Theorem~12.5]{GH}, and \cite[Theorem~4.3, Appendix]{Ell2}.

\end{proof}

\begin{remark}
The equivalence $(3) \Longleftrightarrow (4)$ of Theorem \ref{AF-equivalences-thm} is one consequence of a deep result of Elliott (\cite[Theorem 4.3]{Ell2}).  This connection between the  structure of a dense subalgebra $R$ of an AF-algebra $A$ viewed as an \emph{algebra}, versus the structure of $R$ viewed simply as a \emph{ring}, is for us a crucial bridge between the analytic and algebraic sides of our investigation.
\end{remark}

We are now in position to present a solution to the Isomorphism Conjecture for Graph Algebras in case the graphs are acyclic.

\begin{proposition} \label{ultramatricial-iso-prob}
Let $E$ and $F$ be graphs with no cycles.  Then the following are equivalent.
\begin{enumerate}
\item $C^*(E) \cong C^*(F)$ (as $*$-algebras),
\item $L_\C(E) \cong L_\C(E)$ (as $*$-algebras),
\item $L_\C(E) \cong L_\C(F)$ (as algebras), and
\item $L_\C(E) \cong L_\C(F)$ (as rings).
\end{enumerate}
\end{proposition}

\begin{proof}
Since $E$ and $F$ have no cycles, it follows from \cite[Theorem~2.4]{KPR} that $C^*(E)$ and $C^*(F)$ are AF-algebras.  Also, since $E$ and $F$ have no cycles, it follows that $L_\C(E)$ and $L_\C(F)$ are ultramatricial algebras.  (This is shown in \cite[\S3.3 p.106]{W} for row-finite graphs, and an application of desingularization \cite[Theorem~5.2]{AbrPino3} gives the result for countably infinite graphs.  The result is also shown directly, for arbitrary sized graphs, in \cite{AR}.)  Proposition~\ref{embedding-prop} implies that $L_\C(E)$ is isomorphic to a dense $*$-subalgebra of $C^*(E)$ and $L_\C(F)$ is isomorphic to a dense $*$-subalgebra of $C^*(F)$.  The result then follows from Theorem~\ref{AF-equivalences-thm}.
\end{proof}

\begin{remark}
Proposition~\ref{ultramatricial-iso-prob} shows that the Isomorphism Conjecture for Graph Algebras of \S\ref{Iso-Prob-sec}, Conjecture 1 of \S\ref{alg-and-*-alg-hom-sec}, and Conjecture 2 of \S\ref{alg-and-*-alg-hom-sec} all have affirmative answers in the case that $E$ and $F$ have no cycles.
\end{remark}

\section{Isomorphisms of simple graph algebras} \label{simple-sec}

In this section we prove that the Isomorphism Conjecture for Graph Algebras has an affirmative answer when the graph algebras are simple and come from row-finite graphs.  To do this we will need to make use of $K$-theory and the Kirchberg-Phillips Classification Theorem for purely infinite, simple, separable, nuclear $C^*$-algebras.  Throughout we let $K_1(R)$ denote the algebraic $K_1$-group of a ring $R$, and we let $K_1^\textnormal{top}$ denote the topological $K_1$-group of a $C^*$-algebra $A$.  Recall that with $K_0$-groups, we have that $K_0(A) \cong K_0^\textnormal{top} (A)$ whenever $A$ is a $C^*$-algebra \cite[Theorem~1.1]{Ros}.

\begin{definition}
If $G$ is an abelian group, then an element $g \in G$ is \emph{divisible by $n$} if there exists $x \in G$ such that $nx =g$.    We call $g$ \emph{divisible} if $g$ is divisible by $n$ for all $n \in \N$.  We say an abelian group $G$ is \emph{divisible} if every element of $G$ is divisible.
\end{definition}

\begin{remark}
If $G$ is a free abelian group, then the only divisible element of $G$ is the identity $0$.  This is because $G$ is isomorphic to a (possibly infinite) direct sum of copies of $\Z$.
\end{remark}

The following result is well-known in the abelian group community; we include a proof for completeness.

\begin{lemma} \label{free-divisible-lem}
Suppose that $D_1$ and $D_2$ are divisible abelian groups and that $F_1$ and $F_2$ are free abelian groups.   If $D_1 \oplus F_1 \cong D_2 \oplus F_2$, then $D_1 \cong D_2$ and $F_1 \cong F_2$.
\end{lemma}

\begin{proof}
Let $\phi : D_1 \oplus F_1 \to D_2 \oplus F_2$ be an isomorphism.  If $d \in D_1$, then since $D_1$ is a divisible group, the element $(d,0)$ is a divisible element of $D_1 \oplus F_1$.  Thus $\phi(d_1, 0)$ is a divisible element of $D_2 \oplus F_2$.  Suppose $\phi(d_1, 0) = (d_2, f_2)$.  Then for any $n \in \N$ there exists $(x,y) \in D_2 \oplus F_2$ such that $n (x,y) = (d_2, f_2)$.  Hence for any $n \in \N$ there exists $y \in F_2$ such that $ny =  f_2$.  Thus $f_2$ is divisible in $F_2$, and since $F_2$ is a free abelian group we may conclude that $f_2 = 0$.  Therefore $\phi (D_1 \oplus 0) \subseteq D_2 \oplus 0$.  A similar argument using $\phi^{-1}$ shows that $\phi (D_1 \oplus 0) =D_2 \oplus 0$.  In particular, we get $D_1 \cong D_2$.  It follows that $(D_1 \oplus F_1) /  (D_1 \oplus 0) \cong (D_2 \oplus F_2) / \phi(D_1 \oplus 0)$, and hence $(D_1 \oplus F_1) /  (D_1 \oplus 0) \cong (D_2 \oplus F_2) / (D_2 \oplus 0)$, and thus $F_1 \cong F_2$.
\end{proof}

 Recall that $\C^\times$ denotes the multiplicative group of nonzero complex numbers.  For a natural number $m \in \N$ let $(\C^\times)^m$ denote the direct sum of $m$ copies of $C^\times$ and let $(\C^\times)^\infty$ denote the direct sum of a countably infinite number of copies of $\C^\times$.

\begin{theorem} \label{K-theory-implications-thm}
Let $E$ and $F$ be row-finite graphs with no sinks.  Then the following two implications hold:
\begin{itemize}
\item[(1)] If $K_0(L_\C(E)) \cong K_0(L_\C(F))$, then $K_0^\textnormal{top} (C^*(E)) \cong K_0^\textnormal{top} (C^*(F))$.  Furthermore, if the isomorphism from $K_0(L_\C(E))$ to $K_0(L_\C(F))$ preserves the (pre)order and scale, then there exists an isomorphism from $K_0^\textnormal{top} (C^*(E))$ to $K_0^\textnormal{top} (C^*(F))$ that preserves the (pre)order and scale.  In addition, if $L_\C(E)$ and $L_\C(F)$ are unital and the isomorphism from $K_0(L_\C(E))$ to $K_0(L_\C(F))$ preserves the class of the unit, then there exists an isomorphism from $K_0^\textnormal{top} (C^*(E))$ to $K_0^\textnormal{top} (C^*(F))$ that preserves the class of the unit.
\item[(2)] If $K_1(L_\C(E)) \cong K_1(L_\C(F))$, then $K_1^\textnormal{top} (C^*(E)) \cong K_1^\textnormal{top} (C^*(F))$.
\end{itemize}
\end{theorem}

\begin{proof}  The result of (1) follows from \cite[Theorem~7.1]{AMP} where it is shown that for a row-finite graph $E$ the inclusion $\iota : L_\C(E) \hookrightarrow C^*(E)$ induces an isomorphism from $K_0(L_\C(E))$ onto $K_0^\textnormal{top} (C^*(E))$.

To obtain (2), let $A_E$ be the vertex matrix of $E$, and consider the matrix $A_E^t-I$ (where $M^t$ denotes the transpose of a matrix $M$).  Since $E$ is a row-finite graph, $A_E$ is a row-finite matrix, and $A_E^t-I$ is a column-finite matrix.  When $|E^0|$ is finite, it follows from \cite[Corollary~7.7]{ABC} that
$$K_1(L_\C(E)) \cong \coker (A_E^t - I : (\C^\times)^n \to (\C^\times)^n) \oplus \ker (A_E^t - I : \Z^n \to \Z^n)$$ where $n := |E^0|$.  In addition, the same formula holds when $E^0$ is infinite, provided we allow $n = \infty$.  (To see this, we must use the computation, also proven in \cite[Corollary~7.7]{ABC}, for the $K_1$ group for the Leavitt path algebra of a finite graph with sinks.  In particular, $K_1(L_\C(E)) \cong \coker (B_E^t - I : (\C^\times)^r \to (\C^\times)^n) \oplus \ker (B_E^t - I : \Z^r \to \Z^n)$ where $B_E$ is the nonsquare matrix obtained from the vertex matrix $A_E$ by deleting the zero rows corresponding to sinks and $r$ is the number of vertices that are not sinks.  One can then take a direct limit representation in terms
of finite graphs, which may involve graphs with sinks even if the
original infinite graph had no sinks. The direct limit
argument then gives the above result in the infinite case, because the vertices in the infinite graph will eventually be nonsinks in some approximation.)

A similar result holds, \emph{mutatis mutandis}, for $F$.  Since $K_1(L_\C(E)) \cong K_1(L_\C(F))$, we have
\begin{align*}
\coker (A_E^t &- I : (\C^\times)^n \to (\C^\times)^n) \oplus \ker (A_E^t - I : \Z^n \to \Z^n) \\ & \cong \coker (A_F^t - I : (\C^\times)^m \to (\C^\times)^m) \oplus \ker (A_F^t - I : \Z^m \to \Z^m)
\end{align*}
where we define $n$ and $m$ to be the (possibly infinite) values $n := | E^0|$ and $m := | F^0|$.

Since $\C^\times$ is a divisible group, it is not hard to show directly that any direct sum $(\C^\times)^k$ is divisible (including the case $k=\infty$).  Because $\coker (A_E^t - I : (\C^\times)^n \to (\C^\times)^n)$ and $\coker (A_F^t - I : (\C^\times)^m \to (\C^\times)^m)$ are quotients of divisible groups, they are themselves divisible groups.  Furthermore, since any direct sum $\Z^k$ is a free abelian group, we see that  $\ker (A_E^t - I : \Z^n \to \Z^n)$ and $\ker (A_F^t - I : \Z^m \to \Z^m)$ are subgroups of free abelian groups, and are therefore free abelian groups themselves (by a standard result in abelian groups).  It follows from Lemma~\ref{free-divisible-lem} that
$$\coker (A_E^t - I : (\C^\times)^n \to (\C^\times)^n) \cong \coker (A_F^t - I : (\C^\times)^m \to (\C^\times)^m)$$
and
$$\ker (A_E^t - I : \Z^n \to \Z^n) \cong \ker (A_F^t - I : \Z^m \to \Z^m).$$
It is proven in \cite[Theorem~7.16]{Rae} that $K_1^\textnormal{top}(C^*(E)) \cong \ker ( A_E^t - I :  \Z^n \to \Z^n)$ and $K_1^\textnormal{top}(C^*(F)) \cong \ker ( A_F^t - I :  \Z^n \to \Z^n)$.  Thus $K_1^\textnormal{top}(C^*(E)) \cong K_1^\textnormal{top}(C^*(F))$.
\end{proof}

We are now in position to establish another case of the Isomorphism Conjecture for Graph Algebras.

\begin{proposition} \label{purely-inf-iso-prob}
Let $E$ and $F$ be row-finite graphs that are each cofinal, satisfy Condition~(L), and contain at least one cycle.  If $L_\C(E) \cong L_\C(F)$ (as rings), then $C^*(E) \cong C^*(F)$ (as $*$-algebras).
\end{proposition}

\begin{proof}
Note that since $E$ and $F$ are each cofinal and contain a cycle, neither graph can contain a sink.  Thus $E$ and $F$ are both row-finite graphs with no sinks.

In addition, since $L_\C(E) \cong L_\C(F)$ (as rings) we have that either $L_\C(E)$ and $L_\C(F)$ are both unital (and hence $C^*(E)$ and $C^*(F)$ are both unital) or $L_\C(E)$ and $L_\C(F)$ are both nonunital (and hence $C^*(E)$ and $C^*(F)$ are both nonunital).

Moreover, since $L_\C(E) \cong L_\C(F)$ (as rings), it follows from previous remarks that $K_0(L_\C(E)) \cong K_0(L_\C(F))$ as groups.  Furthermore, in the unital case this isomorphism may be chosen to preserve the class of the unit (i.e., $[L_\C(E) ] \mapsto [ L_\C(F)]$).  Likewise, since  $L_\C(E) \cong L_\C(F)$ (as rings), it follows that $K_1(L_\C(E)) \cong K_1(L_\C(F))$ as groups.  Theorem~\ref{K-theory-implications-thm} then implies that $K_0^\textnormal{top}(C^*(E)) \cong K_0^\textnormal{top}(C^*(F))$ (and in the unital case this isomorphism may be chosen to preserve the class of the unit), and also $K_1^\textnormal{top}(C^*(E)) \cong K_1^\textnormal{top}(C^*(F))$.

Because $E$ and $F$ are cofinal, satisfy Condition~(L), and contain at least one cycle, it follows from \cite[Theorem~5.1, Remark~5.6]{BPRS} that $C^*(E)$ and $C^*(F)$ are each purely infinite and simple.  Since all graph $C^*$-algebras of countable graphs are separable, nuclear, and satisfy the UCT \cite[Remark~A.11.13]{Tom-thesis}, the hypotheses of the Kirchberg-Phillips Classification Theorem \cite[Theorem~4.2.4]{Phi} are satisfied.  Therefore, $C^*(E) \cong C^*(F)$ (as $*$-algebras).
\end{proof}

The two previously established cases of the Isomorphism Conjecture blend nicely to produce the following result.

\begin{theorem} \label{Iso-Conjecture-simple-alg-thm}
Let $E$ and $F$ be row-finite graphs such that $L_\C(E)$ and $L_\C(F)$ are simple.  If $L_\C(E) \cong L_\C(F)$ (as rings), then $C^*(E) \cong C^*(F)$ (as $*$-algebras).
\end{theorem}

\begin{proof}
It follows from the dichotomy for simple Leavitt path algebras \cite[Theorem~4.4]{AbrPino3} that $L_\C(E)$ and $L_\C(F)$ are either both ultramatricial (in the case that $E$ and $F$ contain no cycles) or are both purely infinite (in the case that $E$ and $F$ each contain at least one cycle).  If $L_\C(E)$ and $L_\C(F)$ are both ultramatricial, then the result follows from Proposition~\ref{ultramatricial-iso-prob}. If $L_\C(E)$ and $L_\C(F)$ are both purely infinite, then $E$ and $F$ are each cofinal graphs satisfying Condition~(L) and containing a cycle, so the result follows from Proposition~\ref{purely-inf-iso-prob}.
\end{proof}

\section{The Morita Equivalence Conjecture for Graph Algebras} \label{Mor-Equiv-sec}

Two unital rings are said to be \emph{Morita equivalent} if and only if their categories of left modules are equivalent if and only if there exists an equivalence bimodule between the rings (see \cite[Ch.6]{AndFul} for details).  The notion of Morita equivalence has also been extended to various classes of nonunital rings (see, for example, \cite{AAM}), including rings with \emph{enough idempotents} (see Definition~\ref{enough-idempotents-def}).  Any Leavitt path algebra is a ring with enough idempotents, and consequently there is a notion of Morita equivalence for the class of Leavitt path algebras.  Given two rings  $R$ and $S$, we will write $R \ME S$ to indicate that $R$ and $S$ are Morita equivalent.

Rieffel was the first to extend the notion of Morita equivalence to $C^*$-algebras, and in the past 25 years these ideas have proven extremely fruitful and developed into a standard set of tools for $C^*$-algebraists (see \cite{Rie82} for details).  When extending the definition from rings to $C^*$-algebras, there are two ways to proceed:  Rieffel defined two $C^*$-algebras to be \emph{Morita equivalent} if their categories of Hermitian modules are equivalent (and this equivalence is also required to preserve the $*$-operation on the morphisms of these categories) \cite[p.295--296]{Rie82}.  Rieffel also defined two $C^*$-algebras to be \emph{strongly Morita equivalent} if there exists an equivalence $C^*$-bimodule between them \cite[p.295--296]{Rie82}.  It turns out that if $A$ and $B$ are $C^*$-algebras, then $A$ is strongly Morita equivalent to $B$ implies that $A$ is Morita equivalent to $B$.  The converse, in general, does not hold.

 In addition, two unital $C^*$-algebras are algebraically Morita equivalent (i.e., Morita equivalent as rings) if and only if they are strongly Morita equivalent \cite[p.295--296]{Rie82}.  In the development of $C^*$-algebras, strong Morita equivalence has emerged as the most useful notion, and that is what most $C^*$-algebraists focus on.  We will do the same here, and for two $C^*$-algebras $A$ and $B$ we write $A \SME B$ to indicate that $A$ and $B$ are strongly Morita equivalent.  (We warn the reader that since strong Morita equivalence is the predominant notion in much of $C^*$-algebra theory, in  the more recent literature many authors have taken to simply writing ``Morita equivalent" to mean strongly Morita equivalent, an inconsistency with Rieffel's early definition of Morita equivalences.)

In this section we examine Morita equivalence for graph algebras, and consider the following conjecture, which we call the Morita Equivalence Conjecture for Graph Algebras.

$ $

\noindent \textbf{The Morita Equivalence Conjecture for Graph Algebras:} If $E$ and $F$ are graphs, then $\LC(E) \ME \LC(F)$ implies that $C^*(E) \SME C^*(F)$.

$ $

In this section we shall prove an affirmative answer to the Morita Equivalence Conjecture for Graph Algebras for simple graph algebras coming from row-finite graphs as well as for graph algebras of graphs with no cycles.  We also prove that an affirmative answer to the Isomorphism Conjecture for Graph Algebras (see \S\ref{Iso-Prob-sec}) for all graphs implies an affirmative answer for all graphs to the Morita Equivalence Conjecture for Graph Algebras.  We accomplish this by considering stabilizations of graphs and their associated algebras.

\begin{definition}  Given a graph $E$, let $M_nE$ be the graph formed from $E$ by taking each $v \in E^0$ and attaching a ``head" of length $n-1$ of the form
\begin{equation*}
\xymatrix{  v_{n-1} \ar[r]^{e^v_{n-1}} & \ldots  \ar[r]^{e^v_3} & v_2 \ar[r]^{e^v_2} & v_1 \ar[r]^{e^v_1} & v }
\end{equation*}
to $E$.
\end{definition}

\begin{example}
If $E$ is the graph
$$
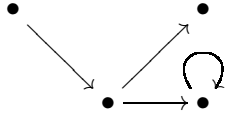
\xymatrix{
\bullet \ar[rd] & & \bullet \\ & \bullet \ar[ru] \ar[r] & \bullet \ar@(ul,ur)
}
$$
then $M_3E$ is the graph
$$
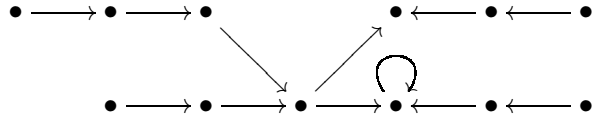
\xymatrix{
 \bullet \ar[r] & \bullet \ar[r] & \bullet \ar[rd] & & \bullet
& \bullet \ar[l] & \bullet \ar[l] \\
& \bullet \ar[r] & \bullet \ar[r] & \bullet
\ar[ru] \ar[r] & \bullet \ar@(ul,ur) & \bullet \ar[l] & \bullet \ar[l]  }
$$
\end{example}

\noindent The following result generalizes \cite[Proposition 13]{AbrPino2} to all graph algebras.

\begin{proposition} \label{Mn-graphs-prop}
If $E$ is a graph, then for any $n\in\N$ it is the case that $C^*(M_nE) \cong {\rm M}_n(C^*(E))$ (as $*$-algebras) and $\LC(M_nE) \cong {\rm M}_n(\LC(E))$ (as $*$-algebras).  In addition, if $K$ is any field, then $L_K(M_nE) \cong {\rm M}_n(L_K(E))$ (as algebras).
\end{proposition}

\begin{proof}
We will first prove that $C^*(M_nE) \cong {\rm M}_n(C^*(E))$ (as $*$-algebras).  For
$1\leq i,j \leq n$, we let $E_{i,j}$ denote the matrix in
${\rm M}_n(\C)$ with a $1$ in the $(i,j)$\textsuperscript{th} position
and $0$'s elsewhere.   For $a \in C^*(E)$ we let $E_{i,j}
\otimes a$ denote the matrix in ${\rm M}_n(C^*(E))$ with
$a$ in the $(i,j)$\textsuperscript{th} position and $0$'s
elsewhere. Note that $(E_{i,j} \otimes a)(E_{k,l} \otimes b ) =
E_{i,j} E_{k,l} \otimes ab$ in ${\rm M}_n(C^*(E))$.

Let $\{s_e, p_v : e\in E^1, v \in E^0 \}$ be a generating
Cuntz-Krieger $E$-family for $C^*(E)$. For each $v \in E^0$ and $e
\in E^1$ define $$P_v := E_{1,1} \otimes p_v \qquad \text{ and }
\qquad S_e := E_{1,1} \otimes s_e.$$  Also for $v \in E^0$ and $k
\in \{1, \ldots, n-1\}$ define $$P_{v_k} := E_{k+1, k+1} \otimes
p_v \qquad \text{ and } \qquad S_{e^v_k} := E_{k+1, k} \otimes
p_v.$$

It is straightforward to verify that $\{ S_e, S_{e^v_k} : v \in E^0, e \in E^1, 1 \leq k \leq n-1 \} \cup \{P_v, P_{v_k} : v \in E^0, 1 \leq k \leq n-1 \}$ is a Cuntz-Krieger $M_nE$-family in ${\rm M}_n(C^*(E))$. Thus there exists a canonical algebra $*$-homomorphism $\phi : C^*(M_nE) \to
{\rm M}_n(C^*(E))$.  To see that $\phi$ is onto, it suffices to show
that for all $v \in E^0$,  $e \in E^1$, and $1 \leq i,j \leq n$ it
is the case that $E_{i,j} \otimes s_e$ and $E_{i,j} \otimes p_v$
are in the $*$-subalgebra of ${\rm M}_n(C^*(E))$ generated by $\{ S_e, P_v :
e \in M_nE^1, v \in M_nE^0 \}$.  For $i = j$ we have $$E_{i,j}
\otimes p_v = E_{i,i} \otimes p_v = \begin{cases} P_v & \text{ if $i=1$} \\ P_{v_{i-1}} & \text{ if $i \geq 2$},\end{cases}$$ for $i > j$ we have
$$E_{i,j} \otimes p_v = (E_{i,i-1} \otimes p_v) (E_{i-1,i-2}
\otimes p_v) \ldots (E_{j+1, j} \otimes p_v) = S_{e_{i-1}^v}
S_{e_{i-2}^v} \ldots S_{e_j^v},$$ and for $i < j$ we have
$$E_{i,j} \otimes p_v = (E_{i,i+1} \otimes p_v) (E_{i+1,i+2}
\otimes p_v) \ldots (E_{j-1, j} \otimes p_v) = S_{e_i^v}^*
S_{e_{i+1}^v}^* \ldots S_{e_{j-1}^v}^*.$$  Finally, for any $e \in
E^1$ and $1 \leq i,j \leq n$ we have $$E_{i,j} \otimes s_e = (E_{i,1} \otimes
p_{s(e)}) (E_{1,1} \otimes s_e ) (E_{1,j} \otimes p_{r(e)}) =
S_{e_{i-1}^v} \ldots S_{e_1^v} S_e S_{e_{1}^v}^* \ldots
S_{e_{j-1}^v}^*.$$  Thus $\phi$ is onto.

To see that $\phi$ is injective, we use the Gauge-Invariant
Uniqueness Theorem \cite[Theorem~2.1]{BHRS}. Let $\gamma^E$ be the
canonical gauge action on $C^*(E)$, and let $\beta$ be the gauge
action on ${\rm M}_n(\C)$ defined by $\beta_z (E_{i,j}) =
z^{i-j}E_{i,j}$. Then there is a gauge action $\gamma$ on ${\rm M}_n(
C^*(E))$ given by $\gamma_z (E_{i,j} \otimes a) = \beta_z(E_{i,j})
\gamma_z^E(a)$.  If one lets $\gamma^{M_nE}$ denote the canonical
gauge action on $C^*(M_nE)$, then one can verify that $\phi \circ
\gamma^{M_nE} = \gamma \circ \phi$. (Simply check the equality on
generators.) Since the $P_v$'s are also nonzero, the
Gauge-Invariant Uniqueness Theorem implies that $\phi$ is
injective.  Thus $\phi$ is an algebra $*$-isomorphism, and $C^*(M_nE) \cong
{\rm M}_n(C^*(E))$ (as $*$-algebras).

The same proof can be applied \emph{mutatis mutandis} (and, in
particular, by applying the Graded Uniqueness Theorem
\cite[Theorem~4.8]{Tom10} in place of the Gauge-Invariant
Uniqueness Theorem) to show that $\LC(M_nE) \cong {\rm M}_n(\LC(E))$ (as $*$-algebras) and that $L_K(M_nE) \cong {\rm M}_n(L_K(E))$ (as algebras), for any field $K$.  We mention that specifically we are using the following grading on the matrix ring: If $x \in L_K(E)_t$, then the degree of the element in ${\rm M}_n(L_K(E))$ having $x$ in the $(i,j)$ entry and zeros elsewhere is equal to $t+(i-j)$.
\end{proof}

\begin{definition}\label{stabilization}  Given a graph $E$, let $SE$ be the graph formed from $E$ by taking each $v \in E^0$ and attaching an infinite ``head" of the form
\begin{equation*}
\xymatrix{   \ldots  \ar[r]^{e^v_4} & \ar[r]^{e^v_3} v_3& v_2 \ar[r]^{e^v_2} & v_1 \ar[r]^{e^v_1} & v }
\end{equation*}
to $E$.  We call $SE$ the \emph{stabilization} of $E$.
\end{definition}

\begin{example}
If $E$ is the graph
$$
\xymatrix{
\bullet \ar[rd] & & \bullet \\ & \bullet \ar[ru] \ar[r] & \bullet \ar@(ul,ur)
}
$$
then $SE$ is the graph
$$
\xymatrix{
\cdots \ar[r] & \bullet \ar[r] & \bullet \ar[r] & \bullet \ar[rd] & & \bullet
& \bullet \ar[l] & \bullet \ar[l] & \cdots \ar[l] \\
\cdots \ar[r] & \bullet \ar[r] & \bullet \ar[r] & \bullet \ar[r] & \bullet
\ar[ru] \ar[r] & \bullet \ar@(ul,ur) & \bullet \ar[l] & \bullet \ar[l] &
\cdots \ar[l] }
$$
\end{example}

\begin{definition}\label{M-infinity}
For any ring $R$ we let ${\rm M}_\infty(R)$ denote the ring of finitely supported, countably infinite square matrices with coefficients in $R$.    If $R$ is an algebra (respectively, a $*$-algebra), then ${\rm M}_\infty(R)$ inherits an algebra structure (respectively, a $*$-algebra structure) in the usual way.
\end{definition}

In general, if $G$ is a subgraph of the graph $E$ then the inclusion $G \hookrightarrow E$ does not necessarily induce homomorphisms $\rho_L:L_K(G) \to L_K(E)$ and  $\rho_C: C^*(G)\to C^*(E)$.  This is due to the fact that the Cuntz-Krieger relation at a vertex $v$, when $v$ is viewed in $G$, need not be compatible with the corresponding Cuntz-Krieger relation at $v$, when $v$ is viewed in the larger graph $E$.   However, in certain situations the inclusion $G \hookrightarrow E$ does induce natural homomorphisms  $\rho_L:L_K(G) \to L_K(E)$ and  $\rho_C: C^*(G)\to C^*(E)$, which are necessarily injective, so that $L_K(G)$ can be viewed as a subalgebra of $L_K(E)$, and $C^*(G)$ can be viewed as a $C^*$-subalgebra of $C^*(E)$.   Specifically, this happens when $G$ is a \emph{complete} subgraph of $E$.

 \begin{definition}
A subgraph $G$ of a graph $E$ is called {\it complete} in case, for each regular vertex $v\in G^0$, we have  $s^{-1}_G(v) = s^{-1}_E(v)$.  (In other words, a subgraph $G$ of $E$ is complete if, whenever $v\in G^0$ emits a nonzero, finite number of edges in $G$, then necessarily the subgraph $G$ contains all of the edges in $E$ emitted by $v$.)
\end{definition}

\begin{proposition} \label{suspension-prop}
If $E$ is a graph, then
\begin{itemize}
\item[(1)] $\LC(SE) \cong {\rm M}_{\infty}(\LC(E))$ (as $*$-algebras),
\item[(2)] $L_K(SE) \cong {\rm M}_{\infty}(L_K(E))$ (as $K$-algebras) for any field $K$, and
\item[(3)] $C^*(SE) \cong C^*(E) \otimes \K$ (as $*$-algebras), where $\K = \K(H)$ denotes the compact operators on a separable infinite-dimensional Hilbert space $H$.
\end{itemize}
\end{proposition}

\begin{proof}

For each $n$ we see that $M_nE$ sits as a complete subgraph of $M_{n+1}E$.  Thus the inclusion $i_n : M_nE \hookrightarrow M_{n+1}E$ induces an algebra $*$-homomorphism $(i_n)_* : {\rm M}_n(\LC(E)) \hookrightarrow {\rm M}_{n+1} (\LC(E))$ taking $A \mapsto \left( \begin{smallmatrix} A & 0 \\ 0 & 0 \end{smallmatrix} \right)$.
We also see that $SE = \bigcup_{n=1}^\infty M_nE$.  Furthermore, since the functor $E \mapsto \LC(E)$ is continuous with respect to algebraic direct limits \cite[Lemma~3.2]{AMP}, we have that $\LC(SE)$ is isomorphic (as a $*$-algebra) to the algebraic direct limit $\bigcup_{n=1}^\infty \LC(M_nE)$ and using Proposition~\ref{Mn-graphs-prop} we obtain $\bigcup_{n=1}^\infty \LC(M_nE) = \bigcup_{n=1}^\infty {\rm M}_n(\LC(E)) = {\rm M}_{\infty} (\LC(E))$.  Thus $\LC(SE) \cong {\rm M}_{\infty} (\LC(E))$ (as $*$-algebras).  A similar argument with $K$ in place of $\C$, and using algebra homomorphisms, shows that (2) holds.

The result in (3) is proven in \cite[Theorem~4.2]{Tom7}, using the characterization of stability of graph $C^*$-algebras obtained in \cite[Theorem~3.2]{Tom7}.  However, we can give a short direct proof as in the previous paragraph.  Since $SE =  \bigcup_{n=1}^\infty M_nE$, and the functor $E \mapsto C^*(E)$ is continuous with respect to $C^*$-algebraic direct limits \cite[Lemma~3.3]{AMP}, $C^*(SE)$ is isomorphic (as a $*$-algebra) to the $C^*$-algebraic direct limit $\overline{\bigcup_{n=1}^\infty C^*(M_nE)} = \overline{\bigcup_{n=1}^\infty {\rm M}_n(C^*(E))} = C^*(E) \otimes \K$.  Thus $C^*(SE) \cong C^*(E) \otimes \K$ (as $*$-algebras).
\end{proof}

The following definition was first given in \cite{Ful}.

\begin{definition} \label{enough-idempotents-def}
A ring $R$ is said to have a set of \emph{enough idempotents} in case there exists a set of orthogonal idempotents $\{e_i | i\in I\}$ in $R$ for which ${}_RR = \oplus _{i\in I} Re_i$.
\end{definition}

In particular, any unital ring is a ring with enough idempotents, with set $\{1\}$. It is then clear that for any graph $E$, the Leavitt path algebra $L_K(E)$ has a set of enough idempotents, specifically, the set of vertices $E^0$.

The following result provides a bridge between isomorphism and Morita equivalence.  The proof in the unital case is Stephenson's Theorem on infinite matrix rings   \cite[Theorem~3.6]{Ste} (with an explicit description of the isomorphism given in \cite[Lemma 1.2]{A1}).  In the case that $A$ and $B$ each have countable sets of enough idempotents, the proof in this more general setting is given in \cite[Theorem 5 and Remarks~1 and 2, p. 412]{AAM}.

\begin{proposition}  \label{Mor-equiv-iso-mat-prop}
Let $A$ and $B$ be rings with countable sets of enough idempotents.   Then $A \ME B$ if and only if ${\rm M}_{\infty} (A) \cong {\rm M}_{\infty} (B)$ (as rings).
\end{proposition}

\begin{corollary} \label{Mor-equiv-LPAs}
Let $E$ and $F$ be graphs.     Then $L_K(E) \ME L_K(F)$ if and only if ${\rm M}_{\infty}(L_K(E)) \cong {\rm M}_{\infty}(L_K(F))$ (as rings).
\end{corollary}

\begin{proof}
By our standing hypothesis (see Remark~\ref{countable-hypothesis-rem}), the vertex sets $E^0$ and $F^0$ are countable. Thus the algebras $L_K(E)$ and $L_K(F)$ have countable sets of enough idempotents.  Now apply Proposition~\ref{Mor-equiv-iso-mat-prop}.
\end{proof}

In the following two subsections we give two important applications of Proposition~\ref{Mor-equiv-iso-mat-prop}.  First, we give Morita equivalence results for ultramatricial algebras and AF $C^*$-algebras paralleling the isomorphism results of \S \ref{ultramatricial-sec}.  These results are interesting in their own right, and they also imply that the Morita Equivalence Conjecture for Graph Algebras holds for graphs with no cycles.  Second, we show that if $\mathcal{C}$ is a class of graphs for which the Isomorphism Conjecture holds and which is closed under the stabilization operation given in Definition \ref{stabilization},  then the Morita Equivalence Conjecture holds for all graphs in $\mathcal{C}$ as well.  This implies, in particular, that the Morita Equivalence Conjecture for Graph Algebras holds whenever the algebras are simple and come from row-finite graphs, and it also shows that an affirmative answer to the Isomorphism Conjecture for Graph Algebras implies an affirmative answer to the Morita Equivalence Conjecture for Graph Algebras.

\subsection{Morita equivalence of ultramatricial graph algebras}

In this subsection we prove Morita equivalence versions of the results in \S\ref{ultramatricial-sec}.

\begin{theorem} \label{AF-Morita-equivalences-thm} (cf. Theorem~\ref{AF-equivalences-thm})
Let $A$ and $B$ be AF $C^*$-algebras, let $R$ be a dense ultramatricial $*$-subalgebra of $A$, and let $S$ be a dense ultramatricial $*$-subalgebra of $B$.  Then the following are equivalent:
\begin{enumerate}
\item $A \SME B$,
\item $R \ME S$,
\item $(K_0(A), K_0(A)^+) \cong  (K_0(B), K_0(B)^+)$, and
\item $(K_0(R), K_0(R)^+) \cong  (K_0(S), K_0(S)^+)$.
\end{enumerate}
\end{theorem}

\begin{proof}  The result follows from the following equivalences.

$(1) \Longleftrightarrow (3)$ is part of Elliott's Theorem for AF $C^*$-algebras \cite[Theorem~12.1.3]{WO}.

$(3) \Longleftrightarrow (4)$ follows from Lemma~\ref{alg-K-iso-op-K}.

$(1) \Longleftrightarrow (2)$ is obtained via the following argument.  We first note that $A \otimes \K = \overline{{\rm M}_{\infty}(A)}$, and since $R$ is a dense $*$-subalgebra of $A$, it follows that ${\rm M}_{\infty} (R)$ is a dense $*$-subalgebra of $A \otimes \K$.  Similarly ${\rm M}_{\infty} (S)$ is a dense $*$-subalgebra of $B \otimes \K$.  Furthermore, since $A$ and $B$ are AF-algebras, each contains a countable approximate identity and the hypotheses of the Brown-Green-Rieffel Theorem \cite[Theorem~1.2]{BGR} are satisfied. This Theorem then gives $A \SME B$ if and only if $A \otimes \K \cong B \otimes \K$ (as $*$-algebras).  In addition, $A \otimes \K \cong B \otimes \K$ (as $*$-algebras) if and only if ${\rm M}_{\infty} (R) \cong {\rm M}_{\infty} (S)$ (as rings) by [(1) $\Longleftrightarrow$ (4)] of Theorem~\ref{AF-equivalences-thm}.  Finally, ${\rm M}_{\infty} (R) \cong {\rm M}_{\infty} (S)$ (as rings) if and only if $R \ME S$ by Proposition~\ref{Mor-equiv-iso-mat-prop}.
\end{proof}

\begin{proposition} \label{ultramatricial-Mor-equiv-prob}
Let $E$ and $F$ be graphs with no cycles.  Then the following are equivalent.
\begin{enumerate}
\item $C^*(E) \SME C^*(F)$, and
\item $L_\C(E) \ME L_\C(F)$.
\end{enumerate}

\noindent In particular,  the Morita Equivalence Conjecture for Graph Algebras holds for the class of graphs with no cycles.

\end{proposition}

\begin{proof}
Since $E$ and $F$ have no cycles, it follows from \cite[Theorem~2.4]{KPR} that $C^*(E)$ and $C^*(F)$ are AF-algebras.  Also, as noted in the proof of Proposition \ref{ultramatricial-iso-prob}, since $E$ and $F$ have no cycles it follows that $L_\C(E)$ and $L_\C(F)$ are ultramatricial algebras.     Proposition~\ref{embedding-prop} implies that $L_\C(E)$ is isomorphic to a dense $*$-subalgebra of $C^*(E)$ and $L_\C(F)$ is isomorphic to a dense $*$-subalgebra of $C^*(F)$.  The result then follows from Theorem~\ref{AF-Morita-equivalences-thm}.
\end{proof}

\subsection{Morita equivalence and classes closed under stabilization}

We now show that if we have a class of graphs for which: (1) the Isomorphism Conjecture for Graph Algebras has an affirmative answer for all graphs in the class, and (2) the class is closed under stabilization, then the Morita Equivalence Conjecture for Graph Algebras has an affirmative answer as well for all graphs in that class.

\begin{theorem} \label{class-iso-closed-stable-then-ME-thm}
Let $\mathcal{C}$ be a collection of graphs with the following two properties:
\begin{enumerate}
\item[(1)] If $E$ and $F$ are graphs in $\mathcal{C}$, then $L_\C(E) \cong L_\C(F)$ (as rings) implies that $C^*(E) \cong C^*(F)$ (as $*$-algebras), and
\item[(2)] If $E$ is a graph in $\mathcal{C}$, then $SE$ is also in $\mathcal{C}$.
\end{enumerate}
Then for any $E$ and $F$ in $\mathcal{C}$ it is the case that $L_\C(E) \ME L_\C(F)$ implies that $C^*(E) \SME C^*(F)$.
\end{theorem}

\begin{proof}
Let $E$ and $F$ be in $\mathcal{C}$ with $L_\C(E) \ME L_\C(F)$.  Then Proposition~\ref{Mor-equiv-iso-mat-prop} implies that ${\rm M}_{\infty} (L_\C(E)) \cong {\rm M}_{\infty} (L_\C(F))$ (as rings).  It follows from Proposition~\ref{suspension-prop} that $L_\C(SE) \cong L_\C(SF)$ (as rings).  By Condition~(2) on $\mathcal{C}$ we have that $SE$ and $SF$ are in $\mathcal{C}$, and then by Condition~(1) on $\mathcal{C}$ we have that $C^*(SE) \cong C^*(SF)$ (as $*$-algebras).  It then follows from Proposition~\ref{suspension-prop} that $C^*(E) \otimes \K \cong C^*(F) \otimes \K$ (as $*$-algebras).  Since $C^*(E)$ and $C^*(F)$ are stably isomorphic, it follows that $C^*(E) \SME C^*(F)$ \cite[Corollary~3.39]{RW}.
\end{proof}

We note that the only implication used in the above proof that is in fact not a biconditional is the statement that  $\LC (SE) \cong \LC (SF)$ (as rings) implies that $C^*(SE) \cong C^*(SF)$ (as $*$-algebras).  (In particular, the converse of the final implication follows from the Brown-Green-Rieffel Theorem \cite[Theorem~1.2]{BGR}.)

\begin{corollary}
If the Isomorphism Conjecture for Graph Algebras (see \S\ref{Iso-Prob-sec}) is true for all graphs, then the Morita Equivalence Conjecture for Graph Algebras (see the beginning of \S\ref{Mor-Equiv-sec}) holds for all graphs.
\end{corollary}

\begin{corollary} \label{simple-graph-have-two-props}
Let $E$ and $F$ be row-finite graphs such that $L_\C(E)$ and $L_\C(F)$ are simple.  If $L_\C(E) \ME L_\C(F)$, then $C^*(E) \SME C^*(F)$.
\end{corollary}

\begin{proof}
Let
$$\mathcal{C} := \{ E : \text{$E$ is a row-finite graph and $L_\C(E)$ is simple} \}.$$
  It follows from Theorem~\ref{Iso-Conjecture-simple-alg-thm} that if $E$ and $F$ are in $\mathcal{C}$, then $L_\C(E) \cong L_\C(F)$ (as rings) implies $C^*(E) \cong C^*(F)$ (as $*$-algebras).  Furthermore, if $E$ is a graph in $\mathcal{C}$, then $E$ is row-finite, cofinal, satisfies Condition~(L), and has the property that every vertex can reach every sink.  By considering the definition of the stabilization, we see that $SE$ is also row-finite, cofinal, satisfies Condition~(L), and has the property that every vertex can reach every sink.  Thus $L_\C(SE)$ is simple by \cite[Theorem 3.1]{AbrPino3}, and $SE$ is in $\mathcal{C}$.  The result then follows from Theorem~\ref{class-iso-closed-stable-then-ME-thm}.
\end{proof}

\begin{remark}
Here is an alternate way to see that the class $\mathcal{C}$ of Corollary~\ref{simple-graph-have-two-props} is closed under stabilization: Straightforward computations with matrices show that for any ring $R$ with enough idempotents, one has that $R$ is simple if and only if ${\rm M}_\infty(R)$ is simple.  Thus $\mathcal{C}$ is closed under stabilization by Proposition~\ref{suspension-prop}.
\end{remark}

\section{Converses to the Isomorphism Conjecture for Graph Algebras} \label{converses-sec}

In this final section we identify two important classes of graphs for which the converse of the Isomorphism Conjecture for Graph Algebras holds. (We note that there are no known counterexamples to the converse of this conjecture.)  For the first class, we have already seen that the converse holds for graphs with no cycles (see Proposition~\ref{ultramatricial-iso-prob}).     Here is the second.

\begin{theorem}\label{converseisoconjecture}
Let $E$ and $F$ be finite graphs for which
 $C^*(E)$ and $C^*(F)$ are purely infinite simple, and for which ${\rm det}(I - A_E^t)$ and ${\rm det}(I - A_F^t)$ have the same sign (where $A_E$ and $A_F$ denote the vertex matrices of $E$ and $F$, respectively).   If $C^*(E) \cong C^*(F)$ (as $*$-algebras), then $L_\C(E) \cong L_\C(F)$ (as rings; in fact, as $*$-algebras).
\end{theorem}

\begin{proof}
For any finite graph $E$, $C^*(E)$ is purely infinite simple (as a C$^*$-algebra) if and only if $L_\C(E)$ is purely infinite simple (as a ring), since, as indicated previously,  each condition is equivalent to $E$ being cofinal, satisfying Condition (L), and containing at least one cycle.  (See e.g. \cite[Theorem 11]{AbrPino2} and \cite[Proposition~5.1 and Proposition~5.3]{BPRS}.)  If $\phi: C^*(E) \rightarrow C^*(F)$ is a $*$-algebra isomorphism, then  $K_0(\phi) :
K_0(C^*(E)) \to K_0(C^*(F))$ is an isomorphism of groups for which  $K_0(\phi) ( [1_{C^*(E)}] ) =
[1_{C^*(F)}]$.  It then follows from  \cite[Theorem~7.1]{AMP} that
there is an isomorphism $\psi : K_0(\LC(E)) \to
K_0(\LC(F))$ with $\psi ( [1_{\LC(E)}] ) = [1_{\LC(F)}]$.
   Now  \cite[Corollary 2.7]{ALPS} applies to yield the result.

\end{proof}

 We note that \cite[Theorem 2.5]{ALPS}, and the resulting \cite[Corollary 2.7]{ALPS} used in the previous proof, follow from deep results in symbolic dynamics, see e.g. \cite{Franks} and \cite{Huang}.
 
 We close the article by presenting examples of specific classes of graphs to which Theorem \ref{converseisoconjecture} applies.

\begin{example}[Matrix rings over Leavitt and Cuntz algebras]   For positive integers $n$ and $k$ with $n\geq 2$, let  $R^k_n$ be the graph

$$
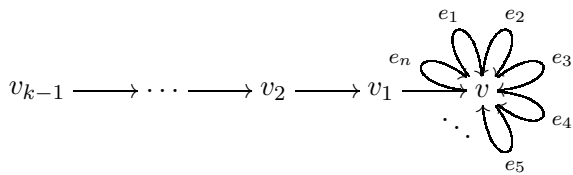
\xymatrix{v_{k-1}\ar[r] & \ldots \ar[r] & v_{2} \ar[r] & v_{1} \ar[r] &  v  \ar@(lu,u)^{e_1} \ar@(u,ur)^{e_2}  \ar@(ru,r)^{e_3} \ar@(r,dr)^{e_4}  \ar@(dr,d)^{e_5} \ar@(l,ul)^{e_n}  \ar@{-}[]^{\begin{matrix} \\  \\ \ddots \end{matrix}}  \\
}
$$
(So $R^k_n$ is precisely the graph $M_kR_n$, where $R_n$ is the rose with $n$ petals graph; i.e., the graph with one vertex and $n$ edges.)

\end{example}

\begin{corollary}\label{LeavittimpliesCuntz}
Let $E$ and $F$ be graphs in $\{ R_n^k : n,k \in \N, n\geq 2 \}$.  If $C^*(E) \cong C^*(F)$ (as $*$-algebras), then $\LC(E) \cong \LC(F)$ (as rings; in fact, as $*$-algebras).
\end{corollary}

\begin{proof}
It is well known that $C^*(R_n^k)$ is purely infinite simple.  It is trivial to show  that the graph $E = R_n^k$ has ${\det}(I-A_E^t) = -(n-1)<0$.  Now apply Theorem \ref{converseisoconjecture}.
\end{proof}

\begin{corollary}\label{converseforsmallgraphs}
Let $E$ and $F$ be graphs for which:
 \begin{itemize}
\item $E$ is cofinal, satisfies Condition (L), and has at least one cycle
\item $E$ has no parallel edges
\item $|E^0| \leq 3$.
\end{itemize}
If $C^*(E) \cong C^*(F)$ (as $*$-algebras), then $\LC(E) \cong \LC(F)$ (as rings; in fact, as $*$-algebras).

\end{corollary}

\begin{proof}
It is shown in \cite{AAP2}  that there are (up to isomorphism) three graphs of the given type having $|E^0| = 2$, and thirty-four graphs of the given type having $|E^0|=3$. (There are no such graphs having just one vertex.)  The first indicated set of conditions yield that $C^*(E)$ and $C^*(F)$ are purely infinite simple.   For each of these thirty-seven graphs, one checks that ${\rm det}(I-A_E^t) <0$.   Now apply  Theorem \ref{converseisoconjecture}.
\end{proof}

\end{document}